\documentclass[reqno,12pt]{amsart} 
\usepackage{latexsym}
\usepackage{amsxtra}
\usepackage{verbatim}
\usepackage{color}
\usepackage{amsthm,amsmath,amsfonts,amssymb,amscd,mathrsfs,graphics}
\usepackage{times,float}
\usepackage{appendix}
\usepackage{txfonts}%,pxfonts}
\usepackage{hyperref,supertabular}
\usepackage{ifpdf}
\usepackage{enumerate}
\usepackage{pst-tree}
\usepackage{fancyhdr}

\usepackage[top=1.5in, bottom=1.5in, left=1.5in, right=1.5in]{geometry}

\allowdisplaybreaks

\newtheorem{theorem}{Theorem}[section]
\newtheorem{lemma}[theorem]{Lemma}
\newtheorem{proposition}[theorem]{Proposition}
\newtheorem{corollary}[theorem]{Corollary}

\newcommand{\leftexp}[2]{{\vphantom{#2}}^{{\rm #1}}{#2}}
\newcommand{\leftbase}[2]{{\vphantom{#2}}_{{\rm #1}}{#2}}

\newcommand{\A}{\mathscr{A}}
\newcommand{\D}{\mathcal{D}}
\def\d{\partial}

\def\la{\lambda}
\def\si{\sigma}
\def\Si{\Sigma}
\def\ge{\epsilon}

\newcommand{\q}{\mathbf{q}}

\newcommand{\LD}{\Big\langle}
\newcommand{\RD}{\Big\rangle}

\makeatletter

\newcommand{\Rmnum}[1]{\expandafter\@slowromancap\romannumeral #1@}
\makeatother

\theoremstyle{definition}

\newcommand{\beq}{\begin{equation}}
\newcommand{\eeq}{\end{equation}}
\newcommand{\bqa}{\begin{eqnarray}}
\newcommand{\eqa}{\end{eqnarray}}
\newcommand{\beqa}{\begin{equation*}}
\newcommand{\ben}{\begin{eqnarray*}}
\newcommand{\eeqa}{\end{equation*}}
\newcommand{\een}{\end{eqnarray*}}

\newcommand{\X}{\mathcal{X}}

\newcommand{\B}{\mathcal{B}}
\newcommand{\C}{\mathbb{C}}

\renewcommand{\H}{\mathcal{H}}

\newcommand{\M}{\mathcal{M}}

\renewcommand{\O}{\mathcal{O}}

\renewcommand{\S}{\mathcal{S}}
\renewcommand{\X}{\mathcal{X}}
\newcommand{\T}{\mathcal{T}}

\newcommand{\Z}{\mathbb{Z}}

\newcommand{\s}{\mathbf{s}}

\newcommand{\f}{\mathbf{f}}
\newcommand{\m}{\mathbf{m}}

\renewcommand{\r}{\mathbf{r}}
\renewcommand{\t}{\mathbf{t}}
\newcommand{\x}{\mathbf{x}}
\newcommand{\y}{\mathbf{y}}

\newcommand{\lieh}{\mathfrak{h}}

\title[Modular groups and simple elliptic singularities]{The modular
  group for the total ancestor potential of Fermat simple elliptic
  singularities}
\author{Todor Milanov}
\author{Yefeng Shen}
\address{Kavli Institute for the Physics and Mathematics of the Universe (WPI),
Todai Institutes for Advanced Study, The University of Tokyo, Kashiwa, Chiba 277-8583, Japan}
\email{todor.milanov@ipmu.jp}

\address{Kavli Institute for the Physics and Mathematics of the Universe (WPI),
Todai Institutes for Advanced Study, The University of Tokyo, Kashiwa, Chiba 277-8583, Japan}
\email{yefeng.shen@ipmu.jp}

\begin{document}


\begin{abstract}
In a series of papers \cite{KS,MR}, Krawitz, Milanov, Ruan, and Shen have verified the
so-called Landau-Ginzburg/Calabi-Yau (LG/CY) correspondence for simple
elliptic singularities $E_N^{(1,1)}$ ($N=6,7,8$). As a byproduct it
was also proved that the orbifold Gromov--Witten invariants of the
orbifold projective lines $\mathbb{P}^1_{3,3,3}$,
$\mathbb{P}^1_{4,4,2}$, and  $\mathbb{P}^1_{6,3,2}$ are quasi-modular
forms on an appropriate modular group. While the modular group for
$\mathbb{P}^1_{3,3,3}$ is $\Gamma(3)$, the modular
groups in the other two cases were left unknown. The goal of this
paper is to prove that the modular groups in the remaining two cases
are respectively $\Gamma(4)$ and $\Gamma(6)$.
\end{abstract}
\maketitle
\tableofcontents
\addtocontents{toc}{\protect\setcounter{tocdepth}{1}}

\section{Introduction}
Let $W(\x)=x_1^{a_1}+x_2^{a_2}+x_3^{a_3}$ be a Fermat polynomial whose
exponents $(a_1,a_2,a_3)$ are given by one of the following triples
$(3,3,3)$, $(4,4,2)$, or $(6,3,2)$. Here we use the notation $\x=(x_1,x_2,x_3)$. Since such a polynomial $W$ defines a hypersurface in
$\C^3$ that has a simple-elliptic singularity at $\x={\bf 0}\in\C^3$, we will
sometimes refer to it as an {\em elliptic} Fermat polynomial. We assign
weights $q_i=1/a_i$ to each variable $x_i$, so that $W$ becomes a
quasi-homogeneous polynomial of degree 1.

\subsection{Formulation of the main results}
Let $H=\C[x_1,x_2,x_3]/(W_{x_1},W_{x_2},W_{x_3})$ be the Jacobi
algebra of $W$, $W_{x_i}:=\partial W/\partial x_i$. Given a triple $\r=(r_1,r_2,r_3)$  of non-negative
integers we put $\phi_\r(\x)=x_1^{r_1}x_2^{r_2}x_3^{r_3}.$ We choose a
set $\mathfrak{R}$ of exponents $\r$, s.t., the monomials $\phi_\r(\x)$ project to a
basis of $H$. More presicely, put 
\beq\label{exp}
\mathfrak{R}=\{(r_1,r_2,r_3)\ |\ 0\leq r_i\leq a_i-2 \}.
\eeq
It will be convenient also to decompose $\mathfrak{R}=\{\mathbf{0},\m\}\sqcup
\mathfrak{R}_{\rm tw}$, where $\mathbf{0}:=(0,0,0)$, $\m:=(m_1,m_2,m_3)$, $\mathfrak{R}_{\rm tw}$ corresponds to monomials of
non-integral degree and $\phi_\m(\x)$ is a monomial of degree 1.  
Let us denote by $\Si\subset \C$ the set of all {\em marginal}
deformations 
\ben
f(\si,\x)=W(\x)+\si\,\phi_\m(\x),\quad \si \in \Si,
\een
s.t., $f(\si,\x)$ has only one critical point. The hypersurfaces 
\ben
X_{\si,\la}=\{ \x\in \C^3\ |\ f(\si,\x)=\la\}
\een
form a smooth fibration over $\Si\times (\C\setminus{\{0\}})$, while the
homology (resp. cohomology) groups $H_2(X_{\si,\la};\C)$ (resp. 
$H^2(X_{\si,\la};\C)$) form a vector bundle equipped with a flat
Gauss--Manin connection. We fix a reference point, say $(0,1)$, and
let 
\ben
\lieh=H_2(X_{0,1};\C),\quad \lieh^\vee = H^2(X_{0,1};\C)
\een 
be the reference fibers. The parallel transport around $\la=0$ induces
a monodromy transformation $J\in {\rm GL}(\lieh)$, which commutes with
the monodromy action of $\pi_1(\Si)$ on $\lieh$. In 
other words, we have a monodromy representation
\ben
\rho: \pi_1(\Si)\to {\rm GL}(\lieh_0)\oplus {\rm GL}(\lieh_{\neq 0}),
%\bigoplus_{-1<\alpha\leq 0}\ {\rm GL}(\lieh_\alpha),
\een 
where $\lieh_0$ is the $J$-invariant subspace and $\lieh_{\neq 0}$ is
the direct sum of all eigenspaces of $J$ with eigenvalue $\neq 0$. Put
$\rho=(\rho_0,\rho_{\neq 0})$ and let 
\ben
\overline{\rho}_{\neq 0}: \pi_1(\Si)\longrightarrow 
 {\rm GL}(\lieh_{\neq 0}\, ) /\langle J\rangle,
\een
the map induced from $\rho_{\neq 0}$.  Using an explicit computation
we will check that  
\beq\label{kernels}
{\rm Ker}\ (\rho_0)\subset {\rm Ker}\ (\overline{\rho}_{\neq 0}).
\eeq
It will be nice if one can find a conceptual explanation and determine
if  \eqref{kernels}  is satisfied for other normal forms $W$ of the
simple elliptic singularities. Using \eqref{kernels} we
get an induced homomorphism
\beq\label{rho-W}
\rho_W:\ 
\widetilde{\Gamma}(W) \longrightarrow 
 {\rm GL}(\lieh_{\neq 0}\, ) /\langle J\rangle,
\eeq
where $\widetilde{\Gamma}(W)={\rm Im}(\rho_0).$ 
Put $\Gamma(W):= {\rm Ker }(\rho_W)$.
\begin{theorem}\label{t1}
The total ancestor potential $\A^W_\si(\hbar;\q)$ of the simple elliptic
singularity $W$ transforms as a quasi-modular form on $\Gamma(W)$.
\end{theorem}  
The definition of the total ancestor potential in singularity theory
as well as the precise meaning of the quasi-modularity will be
recalled later on. Our second result can be stated this way.
\begin{theorem}\label{t2}
If $W$ is an elliptic Fermat polynomial of type
$E_6^{(1,1)}$, $E_7^{(1,1)}$, or $E_8^{(1,1)}$,
then $\Gamma(W)$ is respectively $\Gamma(3), \Gamma(4)$, or
$\Gamma(6)$. 
\end{theorem}

\subsection{Applications to Gromov--Witten theory}

The LG/CY correspondence was proposed by Ruan
\cite{Ru1}. In our settings it can be stated this way. A triplet of
non-zero complex numbers $(\la_1, \la_2, \la_3)\in (\C^*)^3$
is called a {\em diagonal symmetry} of $W$ if 
\ben
W(\la_1 x_1, \la_2 x_2, \la_3 x_3)=W(x_1, x_2, x_3).
\een
The diagonal symmetries form a group $G_W$. It contains the element
\ben
J_W={\rm diag}(e^{2\pi \sqrt{-1} q_1},e^{2\pi \sqrt{-1} q_2} , e^{2\pi
  \sqrt{-1} q_3}),\quad q_i=1/a_i.
\een
The equation $W=0$ defines an elliptic curve $X_W$ in the
weighted projective plane $\mathbb{P}^{2}(c_1, c_2, c_3)$, where 
$q_i=c_i/d$ for a common denominator $d.$  The action of the group
$\widetilde{G}_W:=G_W/\langle J_W\rangle$ on $X_W$ is
faithful and the quotient $\mathcal{X}_W:=X_W/\widetilde{G}_W$ is an
orbifold projective line $\mathbb{P}^1_{a_1,a_2,a_3}$. The
LG/CY correspondence predicts that the  GW invariants of
$\mathcal{X}_{W}$ can be obtained from the so-called Fan--Jarvis--Ruan--Witten
(FJRW) invariants (see \cite{FJR,FJR2}) of the pair $(W, G_W)$ via 
analytic continuation and a certain quantizatied symplectic
transformation (c.f. \cite{Ru1}). Chiodo--Ruan addressed the idea of using global mirror symmetry to solve the LG/CY correspondence \cite{CR, CR1}. This approach has been very successful so far, see \cite{CR,KS,MR} for more details.   

The orbifold GW invariants of $\mathcal{X}=\mathbb{P}^1_{a_1,a_2,a_3}$ are defined as follows. Let 
$\overline{\mathcal{M}}_{g,n,\beta}^{\mathcal{X}}$ be the moduli space
of degree-$\beta$ stable maps from a genus-$g$ orbi-curve, equipped with
$n$ marked points, to $\mathcal{X}$. Here $\beta\in {\rm Eff}(\X)$ where ${\rm Eff}(\X)\subset H_2(\mathcal{X};\Z)$ is the cone of effective curve
classes. By definition the Novikov ring is the completed group algebra
of  ${\rm Eff}(\X)$. In our case, since $H_2(\X;\Z)= \Z\cdot[\X],$
where $[\X]$ is the fundamental class of $\X$, we may
identify the Novikov ring with the space of formal power series
$\C[\![q]\!]$ and replace $\overline{\mathcal{M}}_{g,n,d\cdot[\X]}^{\mathcal{X}}$ by $\overline{\mathcal{M}}_{g,n,d}^{\mathcal{X}}$. Let us denote by 
$\pi$ the forgetful map, and by ${\rm
  ev}_i$ the evaluation at the $i$-th marked point
\ben
\begin{array}[c]{ccccc}
\overline{\mathcal{M}}_{g,n}&
\stackrel{\pi}{\longleftarrow}&
\overline{\mathcal{M}}_{g,n,d}^{\mathcal{X}}&
\stackrel{\rm ev_i}{\longrightarrow}&
I\mathcal{X}\ ,
\end{array}
\een
where $I\mathcal{X}$ is the inertia orbifold of $\mathcal{X}$. The moduli space is equipped with a virtual fundamental
cycle $[\overline{\mathcal{M}}_{g,n,d}^{\mathcal{X}}]$, s.t., the maps
\ben
\Lambda_{g,n}^\mathcal{X}:
H^*_{CR}(\X;\C[\![q]\!])^{\otimes n} \longrightarrow
H^*(\overline{\M}_{g,n};\C)
\een
defined by
\ben
\Lambda_{g,n}^\mathcal{X}=\sum_{d=0}^\infty
q^d\,\Lambda_{g,n,d}^\mathcal{X},\quad \Lambda_{g,n,d}^\mathcal{X}(\alpha_1,\dots,\alpha_n):=
\pi_*\, \Big(
[\overline{\mathcal{M}}_{g,n,d}^{\mathcal{X}}]  \cap 
\prod_{i=1}^n {\rm ev}_i^*(\alpha_i)\ 
\Big)
\een
form a CohFT with state space the Chen-Ruan cohomology 
$H^*_{CR}(\X;\C[\![q]\!])$. The {\em ancestor} GW
invariants of $\mathcal{X}$ are by definition the
following formal series:
\beq\label{cor-anc}
\LD
\tau_{k_1}(\alpha_1),\dots,\tau_{k_n}(\alpha_n)\RD_{g,n}=
\sum_{d=0}^\infty\, q^d \,
\int_{\overline{\M}_{g,n}}
\Lambda^\X_{g,n,d}(\alpha_1,\dots
,\alpha_n)\,\psi_1^{k_1}\cdots\psi_n^{k_n}\, ,
\eeq
where $\psi_i$ is the $i$-th psi class on ${\overline{\M}_{g,n}}$, $\alpha_i\in H^*_{CR}(\X;\C),$ and $k_i\in \Z_{\geq 0}$. 
For more details on orbifold Gromov--Witten theory we refer to
\cite{CheR}.
The total ancestor potential of $\X$ is by definition the following
generating series
\ben
\A^\X_q(\hbar;\q) = \exp \Big(\sum_{g,n=0}^\infty \hbar^{g-1} 
\LD \q(\psi_1)+\psi_1,\dots,\q(\psi_n)+\psi_n\RD_{g,n}\Big),
\een
where $\q(z)=\sum_{k=0}^\infty \q_k z^k$ and $\{\q_k\}_{k=0}^\infty$ is
a sequence of formal vector variables with values in $H^*_{\rm CR}(\X;\C)$. The
generating function is a formal series in
$\q_0,\q_1+1,\q_2,\dots$. 

Following the ideas of Krawitz--Shen \cite{KS} and
Milanov--Ruan \cite{MR}, one can obtain a very precise correspondence
between the total ancestor potentials $\A_\si^W$ and $\A^\X_q$ (see
\cite{MS}).  Let us briefly explain this correspondence. Recall that
the curve 
\ben
E_\si=\{f(\si,\x)=0\} \subset \mathbb{P}^2(c_1,c_2,c_3)
\een
is called {\em the elliptic curve at infinity}. Let us think
of $\Si$ as a punctured $\mathbb{P}^1$ and let us select a puncture $p$, s.t., the
$j$-invariant $j(E_\si)\to \infty$ as $\si\to p.$ For example, if $W$
is the Fermat polynomial of type $E_7^{(1,1)}$, then $p=-2,2,$ or
$\infty$. If $W$ is the Fermat polynomial of type $E_8^{(1,1)}$, then
$p$ is a solution to $4p^3+27=0$. The main statement is that there
exists a function $\pi_B(\si)/\pi_{A}(\si)$ (see \cite{MS}) on $\mathbb{P}^1$ holomorphic near $\si=p$,
s.t., under a mirror map $q=\pi_B(\si)/\pi_{A}(\si)$, the total ancestor potential
$\A^\X_q$ coincides with $\A^W_\si$. Let us point out that the
definition of $\A^W_\si$ requires a choice of a {\em primitive form}
in the sense of K. Saito \cite{S1}. Part of the statement is that
there exists a primitive form, s.t., the identification holds. 
Combining the mirror symmetry theorem of \cite{MS} with Theorem \ref{t1} and
Theorem \ref{t2} we get
\begin{corollary} If $\X$ is one of the orbifolds
  $\mathbb{P}^1_{3,3,3}, \mathbb{P}^1_{4,4,2},$ or
$\mathbb{P}^1_{6,3,2}$, then 
the Gromov--Witten invariants \eqref{cor-anc} are quasi-modular forms
respectively on $\Gamma(3), \Gamma(4), $ or $\Gamma(6).$ 
\end{corollary}

\subsection{Acknowledgement}
We thank Yongbin Ruan for his insight and support for this project. 
Both authors would like to thank for many stimulating conversations. The first author benefited from conversations with Satoshi Kondo and Charles Siegel. We thank Arthur
Greenspoon and Noriko Yui for editorial assistance. The work of both authors is
supported by Grant-In-Aid and by the World Premier International
Research Center Initiative (WPI Initiative), MEXT, Japan.

\section{The total ancestor potential in singularity theory}

Let $\S=\Si\times \C^{\mu-1}$. We fix a coordinate system on $\S$,
such that the coordinates of $\s=(s_\m,s_{\r_1},\cdots, s_{\r_{\mu-1}})\in \S$ are indexed by the exponents
$\mathfrak{R}$ (cf. \eqref{exp}) in such a way that $s_\m\in \Si$ and $s_\r\in\C$ for $\r\neq\m$. 
The miniversal deformation of $W$ can be given by the following
function
\beq\label{min-def}
F(\s,\x)=W(\x)+\sum_{\r\in \mathfrak{R}} s_\r\, \phi_\r(\x),
\eeq
where the domain of $F(\s,\x)$ is $X:=\S\times \C^3.$ The marginal
deformations $f(\si,\x)$ are obtained from $F(\s,\x)$ by restricting
$s_\m=\si$ and $s_\r=0$ for $\r\neq\m$. 

It is well known (see \cite{He, SaT}) that Saito's theory of primitive forms (cf. \cite{S1}) gives
rise to a Frobenius manifold structure (cf. \cite{Du}) on $\S.$ In this section
the goal is to recall the key points in the construction of this
Frobenius manifold structure and then use the higher-genus reconstruction
formalism of Givental to define the total ancestor potential of $W$.

\subsection{Saito's theory}
Let $C\subset X$ be the critical variety of $F(\s,\x)$, i.e., the support of the sheaf is
\ben
\O_C:=\O_X/\langle  F_{x_1}, F_{x_2}, F_{x_3}\rangle.
\een
Let $q:X\to \S$ be the projection on the
first factor. The Kodair--Spencer map ($\T_\S$ is the sheaf of
holomorphic vector fields on $\S$) 
\ben
\T_\S\longrightarrow q_*\O_C,\quad \d/\d s_\r\mapsto \d F/\d s_\r\ {\rm
  mod} \ (F_{x_1},F_{x_2},F_{x_3})
\een 
is an isomorphism, which implies that for any $\s\in\S,$ the tangent space $T_\s\S$
is equipped with an associative commutative multiplication
$\bullet_\s$ depending holomorphically on $\s\in \S$. If in
addition we have a volume form $\omega=g(\s,\x)d^3{  \x},$
where $d^3{  \x}=dx_1\wedge dx_2\wedge dx_3$ is the standard volume form; then $q_*\O_C$ (hence $\T_\S$ as well) is equipped with the {\em residue pairing}:
\beq\label{res:pairing}
(\psi_1,\psi_2)= \frac{1}{(2\pi i)^3} \ \int_{\Gamma_\ge} \frac{\psi_1({  \s,\y})\psi_2({  \s,\y})}{F_{y_1}F_{y_2} F_{y_3}}\, \omega,
\eeq
where ${  \y}=(y_1,y_2,y_3)$ is an unimodular coordinate system for the volume form, i.e., $\omega=d^3{  \y}$, and $\Gamma_\ge$ is a real $3$-dimensional cycle supported on $|F_{x_i}|=\ge$ for $1\leq i\leq 3.$

Given a semi-infinite cycle
\beq\label{cycle}
\mathcal{A}\in  \lim_{ \longleftarrow } H_3(\C^3, (\C^3)_{-m} ;\C)\cong \C^\mu,
\eeq
where
\beq\label{level:lower}
(\C^3)_m=\{{  \x}\in \C^3\ |\ {\rm Re}(F({  \s,\x})/z)\leq m\}.
\eeq
Put
\beq\label{osc_integral}
J_{\mathcal{A}}(\s,z) = (-2\pi z)^{-3/2} \, zd_\S \, \int_{\mathcal{A}} e^{F({  \s,\x})/z}\omega,
\eeq
where $d_\S$ is the de Rham differential on $\S$. The oscillatory integrals $J_{\mathcal{A}}$ are by definition sections of the cotangent sheaf $\T_\S^*$.

According to Saito's theory of primitive forms \cite{S1,MoS}, there exists
a volume form $\omega$ such that the residue pairing is flat and the
oscillatory integrals satisfy a system of differential equations,
which in flat-homogeneous coordinates $\t=(t_\r)_{\r\in\mathfrak{R}}$ have the form 
\beq\label{frob_eq1}
z\d_\r J_{\mathcal{A}}({  \t},z) = \d_\r \bullet_{  \t} J_{\mathcal{A}}({  \t},z),
\eeq
where $ \d_\r:=\d/\d t_\r$ and the multiplication is
defined by identifying vectors and covectors via the residue pairing. 
Using the residue pairing, the flat structure, and the
Kodaira--Spencer isomorphism we have the following isomorphisms:
\ben
T^*\S\cong T\S\cong \S\times T_{\bf 0}\S\cong \S\times H.
\een
Due to homogeneity the integrals satisfy a differential equation with respect to the parameter $z\in \C^*$
\beq
\label{frob_eq2}
(z\d_z + E)J_{\mathcal{A}}({  \t},z) =  \Theta\, J_{\mathcal{A}}({ \t},z),
\eeq
where
\ben
E=\sum_{\r\in\mathfrak{R}}d_\r t_\r \d_\r,\quad (d_\r:={\rm deg}\, t_\r={\rm deg}\, s_\r),
\een
is the {\em Euler vector field} and $\Theta$ is the so-called {\em Hodge grading operator }
\ben
\Theta:\T^*_S\rightarrow \T^*_S,\quad \Theta(dt_\r)=\left(\frac{1}{2}-d_\r\right)dt_\r.
\een
The compatibility of the system \eqref{frob_eq1}--\eqref{frob_eq2} implies that the residue pairing, the multiplication, and the Euler vector field give rise to a {\em conformal Frobenius structure} of   conformal  dimension $1$. We refer to B. Dubrovin \cite{Du} for the definition and more details on Frobenius structures.

\subsection{Primitive forms for simple elliptic singularities}\label{sec:pf}

The classification of primitive forms in general is a very difficult
problem. In the case of simple elliptic singularities however, all
primitive forms are known (see \cite{S1}). They are given by
$\omega=d^3\x/\pi_{A}(\sigma)$, where $\pi_A(\si)$ is a period
of the elliptic curve at infinity. Let us recall also that
$\pi_A(\si)$ can be expressed in terms of a period of the
Gelfand-Lerey form $d^3\x/df$ as follows. We embed $\C^3$ in the
weighted projective space $\mathbb{P}^3(1,c_1,c_2,c_3)$ via 
$x_i=X_i/X_0^{c_i}$, $i=1,2,3.$ The Zariski closure of the 
Milnor fiber $X_{\si,1}$ is $\overline{X}_{\si,1}=X_{\si,1}\cup
E_\si$, therefore we have a {\em tube} map
\ben
L:H_1(E_\si;\Z)\to H_2(X_{\si,1};\Z)
\een
which allows us to write
\ben
\pi_A(\si):= \int_{L(A)}\, \frac{d^3x}{df} = 2\pi\sqrt{-1} \int_A \,
{\rm Res}_{E_\si} \frac{d^3x}{df}.
\een
Let us point out that when $\si=0$ the image of the tube map $L$ is
precisely $\lieh_0$ and the monodromy representation $\rho_0$
coincides with the monodromy representation of the elliptic pencil
$E_\si$, $\si\in \Si$.

\subsection{Givental's higher-genus reconstruction formalism}

Following Givental we introduce the vector space
$\H=H((z))$ of formal Laurent series in $z^{-1}$ with
coefficients in $H$, equipped with the symplectic structure
\ben
\Omega(f(z),g(z))= {\rm res}_{z=0} (f(-z),g(z))dz.
\een
Using the polarization $\H=\H_+\oplus \H_-$, where $\H_+=H[z]$ and
$\H_-=H[[z^{-1}]]z^{-1}$ we identify $\H$ with the cotangent bundle
$T^*\H_+$. The goal in this subsection is to define the total ancestor
potential of $W$. 

\subsubsection{The stationary phase asymptotics}\label{asymptotic:sec}
We fix a primitive form $\omega=d^3\x/\pi_A(\si)$ and let $\{t_\r\}$ be flat
coordinates, defined near $\s=0$, s.t., under  the Kodaira--Spencer
isomorphism $T_0\S\cong H$ we have $\d_\r=\phi_\r(\x).$ Since,
$\pi_A$ is a multi-valued analytic function on $\S$, the flat coordinates $t_e$
are also multi-valued analytic functions on $\S$. 

Let $\s\in \S$ be a semi-simple point, i.e., the critical values $\{u_i(\s)\}_{i=1}^{\mu}$
of $F(\s,\x)$ form locally near $\s$ a coordinate system. Let us
also fix a path from $0\in S$ to $\s$, so that we have a fixed branch
of the flat coordinates. Then we have an isomorphism 
\ben
\Psi_\s:\C^\mu\to H,\quad e_i\mapsto \sqrt{\Delta_i}\,\d_{u_i} =
\sqrt{\Delta_i}\, \sum_{\r\in \mathfrak{R}} \frac{\d u_i}{\d t_\r}\, \phi_\r,
\een
where $\Delta_i$ is determined by $(\d/\d u_i,\d/\d
u_j)=\delta_{ij}/\Delta_i$. It is well known that $\Psi_s$
diagonalizes the Frobenius multiplication and the residue pairing,
i.e., 
\ben
e_i\bullet e_j = \sqrt{\Delta_i}e_i\delta_{i,j},\quad (e_i,e_j)=\delta_{ij}.
\een
Let $\S_{\rm ss}$ be the set of all semi-simple points. The complement
$\mathcal{K}=\S\setminus{\S_{\rm ss}}$ is an analytic hypersurface
also known as the {\em caustic}. It corresponds to deformations, s.t.,
$F(\s,\x)$ has at least one non-Morse critical point. By definition, we get a multi-valued analytic map
\ben
\S_{\rm ss}\to {\operatorname{Hom}}_\C(\C^\mu,H),\quad \s\mapsto \Psi_\s.
\een

The system of differential equations (\ref{frob_eq1}) and
(\ref{frob_eq2}) admits a unique formal asymptotical solution of the type 
\ben
\Psi_\s\, R_\s(z)\, e^{U_\s/z},\quad R_\s(z)=1+z\,R_{\s,1}+z^2\,R_{\s,2}+\cdots
\een
where $U_\s$ is a diagonal matrix with entries $u_1(\s),\dots,u_\mu(\s)$ on the
diagonal and $R_{\s,k} \in {\operatorname{Hom}}_\C(\C^\mu,\C^\mu)$.
Alternatively this formal solution coincides with the stationary phase asymptotics of the
following integrals. Let $\B_i$ be the semi-infinite cycle of the
type \eqref{cycle} consisting of all points $\x\in \C^3$ such that the gradient trajectories of $-{\rm Re}(F(\s,\x)/z)$ flow into the critical value $u_i(\s)$. Then
\ben
(-2\pi z)^{-3/2}\ z\,d_S\ \int_{\B_i} e^{F(\s,\x)/z}\omega \ \sim\ e^{u_i(\s)/z}\,\Psi_\s\, R_\s(z) e_i\quad \mbox{ as }\ z\to 0.
\een
We refer to \cite{G1} and \cite{G2} for more details and proofs.

\subsubsection{The quantization formalism}

Let us fix a Darboux coordinate system on $\H$ given by the linear
functions $q_k^\r$, $p_{k,\r}$ defined as follows: 
\ben
\f(z) = \sum_{k=0}^\infty \sum_{\r\in\mathfrak{R}} \ (q_k^{\r}\, \phi_\r\, z^k + p_{k,\r}\,\phi^\r\,(-z)^{-k-1})\quad \in \quad \H,
\een
where $\{\phi^\r\}_{\r\in\mathfrak{R}}$ is a basis of $H$ dual to
$\{\phi_\r\}_{\r\in\mathfrak{R}}$ with respect to the residue pairing.

If ${R}=e^{A(z)}$, where $A(z)$ is an infinitesimal symplectic
transformation, then we define $\widehat{R}$ as follows. Since $A(z)$
is infinitesimal symplectic, the map $\f\in \H \mapsto A\f\in \H$
defines a Hamiltonian vector field with Hamiltonian given by the
quadratic function $h_A(\f) = \frac{1}{2}\Omega(A\f,\f)$. By
definition, the quantization of $e^A$ is given by the differential
operator $e^{\widehat{h}_A},$ where the quadratic Hamiltonians are
quantized according to the following rules: 
\ben
(p_{k',\r'}p_{k'',\r''})\sphat = \hbar\frac{\d^2}{\d q_{k'}^{\r'}\d q_{k''}^{\r''}},\quad
(p_{k',\r'}q_{k''}^{\r''})\sphat = (q_{k''}^{\r''}p_{k',\r'})\sphat = q_{k''}^{\r''}\frac{\d}{\d q_{k'}^{\r'}},\quad
(q_{k'}^{\r'}q_{k''}^{\r''})\sphat =q_{k'}^{\r'}q_{k''}^{\r''}/\hbar.
\een
Note that the quantization defines a projective representation of the Poisson Lie algebra of quadratic Hamiltonians:
\ben
[\widehat{\mathfrak{A}},\widehat{\mathfrak{B}}]=\{\mathfrak{A},\mathfrak{B}\}\sphat + C(\mathfrak{A},\mathfrak{B}),
\een
where $\mathfrak{A}$ and $\mathfrak{B}$ are quadratic Hamiltonians and the values of the cocycle $C$ on a pair of Darboux monomials is non-zero only in the following cases:
\beq\label{cocycle}
C\left(p_{k',\r'}p_{k'',\r''},q_{k'}^{\r'}q_{k''}^{\r''}\right)=
\begin{cases}
1 & \mbox{ if } (k',\r')\neq (k'',\r''),\\
2 & \mbox{ if } (k',\r')=(k'',\r'').
\end{cases}
\eeq

\subsubsection{The total ancestor potential}
By definition, the Kontsevich-Witten tau-function is the following generating series:
\beq\label{D:pt}
\D_{\rm pt}(\hbar;q(z))=\exp\Big( \sum_{g,n}\frac{1}{n!}\hbar^{g-1}\int_{\overline{\M}_{g,n}}\prod_{i=1}^n (q(\psi_i)+\psi_i)\Big),
\eeq
where $q(z)=\sum_k q_k z^k,$ $(q_0,q_1,\ldots)$ are formal variables,
$\psi_i$ ($1\leq i\leq n$) are the first Chern classes of the
cotangent line bundles on $\overline{\M}_{g,n}.$ The function is
interpreted as a formal series in $q_0,q_1+{\bf 1},q_2,\ldots$ whose coefficients are Laurent series in $\hbar$.

Let $\s\in \mathcal{S}_{\rm ss}$ be {\em a semi-simple}
point. Motivated by Gromov--Witten theory Givental introduced the
notion of the {\em total ancestor potential} of a semi-simple
Frobenius structure (see \cite{G1,G2}). In our settings the definition takes the form
\beq\label{ancestor}
\A_\s(\hbar;\q) :=
\widehat{\Psi_\s}\,\widehat{R_\s}\,e^{\widehat{U_\s/z}}\,
\prod_{i=1}^\mu \D_{\rm pt}\left(\hbar\,\Delta_i(\s); \leftexp{\it i}{\bf q} (z) \sqrt{\Delta_i(\s)}\right)
\eeq
where
\ben
\q(z)=\sum_{k=0}^\infty
\sum_{\r\in\mathfrak{R}} q_k^\r\, z^k\phi_\r, \quad 
\leftexp{\it i}{\bf q} (z) = \sum_{k=0}^\infty \leftexp{\it i}{q}_k \,
z^k.
\een
The quantization $\widehat{\Psi_\s}$ is interpreted as the change of variables
\beq\label{change}
\sum_{i=1}^\mu \leftexp{\it i}{\bf q}(z)\,e_i=\Psi_\s^{-1}\q(z)\quad \mbox{i.e.}\quad
\leftexp{\it i}{q}_k\sqrt{\Delta_i} =\sum_{\r\in\mathfrak{R}} \frac{\d u^i}{\d t_\r}\, q_k^\r.
\eeq

\section{Modularity and monodromy}

The flat coordinates are multi-valued analytic functions on $\S$. 
In this section we will compute their monodromy under analytic
continuation. Once this task is completed the proof of Theorem
\ref{t1} will be easy. 

\subsection{Picard--Fuchs equations}\label{sec:PF}
We consider the so-called {\em geometric sections} (see \cite{AGV})
\beq\label{g-sec}
\Phi_\r(\si,\la):= \int\x^\r\,\frac{d^3\x}{df}\quad \in H^2(X_{\si,\la};\C),
\eeq
where $\r=(r_1,r_2,r_3), r_i\in \Z_{\geq 0},$ and
$\x^\r:=\phi_\r=x_1^{r_1}x_2^{r_2}x_3^{r_3}.$ The geometric sections \eqref{g-sec} with all $\r\in\mathfrak{R}$
give rise to a trivialization of the vanishing cohomology
bundle. The Gauss--Manin connection corresponds to a
system of Fuchsian differential equations known as {\em Picard--Fuchs}
equations. It is enough to solve this system at $\la=1$, because the
homogeneity of $f(\si,x)$ yields the following simple relation
\ben
\Phi_\r(\si,\la) = \la^{{\rm deg}(\r)}\, \Phi_\r(\si,1),
\een
where ${\rm deg}(\r):={\rm deg}(\x^\r)=\sum_i r_i\, q_i.$ The Picard--Fuchs
equations have the form
\beq\label{PF}
\d_\si\, \Phi_\r(\si,1) = \sum_{\r'\in\mathfrak{R}} G_{\r,\r'}(\si)\, \Phi_{\r'}(\si,1),
\eeq
where $\mathfrak{G}=\left(G_{\r,\r'}(\si)\right)$ is a square matrix of size $\mu=|\mathfrak{R}|$, whose
entries are holomorphic functions on $\Si$. Let us denote by
$\mathfrak{F}=\left(\mathfrak{F}_{\r',\r''}(\si)\right)$ a fundamental solution to the above system,
i.e., $\mathfrak{F}$ is a non-degenerate matrix satisfying $\d_\si \mathfrak{F} = \mathfrak{G}\cdot \mathfrak{F}$. 
We have
\beq\label{F-matrix}
\Phi_\r(\si,\la) = \la^{{\rm deg}(\r)}\, 
\sum_{\r'\in \mathfrak{R}}\, 
\mathfrak{F}_{\r,\r'}(\si) \,A_{\r'},
\eeq
where $A_{\r'}, \r'\in \mathfrak{R}$ are multi-valued {\em flat} sections. 
The system \eqref{PF} is block-diagonal in the following sense
\ben
G_{\r',\r''}\neq 0\quad \Rightarrow \quad {\rm deg}(\r')- {\rm
  deg}(\r'')\in \Z.
\een
Therefore the matrix $\mathfrak{F}$ is also block-diagonal. Since analytic
continuation around $\la=0$ corresponds to the classical monodromy
transformation $J$, we get that the vectors $\{A_\r\}_{\r\in\mathfrak{R}}$ give
an eigenbasis of $\lieh^\vee$, i.e., $J(A_\r) = e^{-2\pi\sqrt{-1}\, {\rm
    deg}(\r)}A_\r$.  

\subsection{Flat coordinates}
We follow the idea of \cite{MR}, except that we will avoid the use of
explicit formulas. It is convenient to introduce the following
multi-index notation. We will be interested in sequences 
$
\kappa=(\kappa_\r)_{\r\in \mathfrak{R}\setminus{\{\m\}} },
$
where $\kappa_\r$ are non-negative integers. Recall that $d_\r={\rm deg}(s_\r) = 1-{\rm deg}(\x^\r)$, we put
\ben
\x^\kappa := \prod_{\r\in \mathfrak{R}\setminus{\{\m\}}}\, 
(\x^{\r})^{\kappa_\r},\quad 
\frac{\s^\kappa}{\kappa!}:= 
\prod_{\r\in \mathfrak{R}\setminus{\{\m\}}}\,
\frac{s_\r^{\kappa_\r}}{\kappa_{\r}!},\quad 
d_\kappa := {\rm deg}(\s^\kappa)=\sum_{\r\in \mathfrak{R}\setminus{\{\m\}}} \kappa_\r\,d_\r,
\een 

Let us define a block-diagonal matrix $C=(C_{\r,\r'}(s))_{\r,\r'\in\mathfrak{R}}$, whose entries are
holomorphic functions on $\S$
\beq\label{Cer}
C_{\r,\r'}(\s) = \sum_{\kappa:d_\kappa=d_{\r'}} c_{\r,\kappa}(\sigma) \, 
\frac{\s^\kappa}{\kappa!},
\eeq
where the functions $c_{\r,\kappa}(\s)$ are defined from the identity
\beq\label{cek}
(-2\pi)^{-\frac{3}{2}}\, \left(
\int_0^{-\infty} e^\la \la^{{\rm deg}(\x^\kappa)+1} d\la\,\right)\, \int \x^\kappa\,  \frac{d^3\x}{df} = 
\sum_{\r\in\mathfrak{R}} c_{\r,\kappa}(\sigma)\, \Phi_\r(\si,1).
\eeq
The integration path in the first integral is the negative real
axis and the second one is interpreted as a cohomology class in
$H^2(X_{\si,1};\C).$  The identity is obtained by performing
successively integration by parts until the degree of the monomial
$\x^\kappa$ is reduced to some number in the interval $[0,1]$. In
particular, since each integration by parts decreases the degree by an
integer number, the sum on the RHS is over all $\r\in\mathfrak{R}$, s.t.,
$d_\r-d_\kappa\in \Z.$ It follows that the matrix $C$ is
block-diagonal. Given cycles $\alpha_\r\in \lieh^\vee$, $\r\in\mathfrak{R}$, we
define the following multi-valued analytic functions on $\S$:
\ben
t_\m(\s) & = & 
\frac{1}{\pi_A}\left(
C_{0,\m}(\s)\,\LD \Phi_0(\si,1),\alpha_\m\RD\, +
C_{\m,\m}(\s)\,\LD \Phi_m(\si,1),\alpha_\m\RD\, 
\right),\\
t_0(\s) & = & 
\frac{1}{\pi_A}\left(
C_{\m,0}(\s) \,\LD \Phi_\m(\si,1),\alpha_0\RD\, +
C_{0,0}(\s)\,\LD \Phi_0(\si,1),\alpha_0\RD\, \right),\\
t_\r(\s) & = & 
\frac{1}{\pi_A}\left(
\sum_{\r\in\mathfrak{R}: d_{\r'}=d_\r} C_{\r',\r}(\s)\, \LD \Phi_{\r'}(\si,1),\alpha_\r\RD\,
\right),\quad \r\in\mathfrak{R}_{\rm tw}.
\een
Note that by definition $C_{\m,\m}(\s)=0$ and $C_{0,\m}(\s)$ is a constant
independent of $\s$.
\begin{proposition}\label{flat-c}
There are cycles $\{\alpha_\r\}_{\r\in\mathfrak{R}}$ that form an eigenbasis for
the classical monodromy $J$, s.t., 
\begin{itemize}
\item[(i)]
The functions $t_\r(\s)$, $\r\in\mathfrak{R}$
form a flat coordinate system on $\S.$  
\item[(ii)]
We have $\d_0=1$
and $(\d_0,\d_\m)=1$, where $\d_\r:=\d/\d t_\r$.
\item[(iii)] The following identity holds (compare with the definition
  of $t_0(\s)$):
\ben
\frac{1}{2}\sum_{\r',\r''\in\mathfrak{R}} (\d_{\r'},\d_{\r''})\,t_{\r'}t_{\r''}= 
\frac{1}{\pi_A}
\left(
C_{\m,0}(\s) \,\LD \Phi_\m(\si,1),\alpha_\m\RD\, +
C_{0,0}(\s)\,\LD \Phi_0(\si,1),\alpha_\m\RD\right).
\een
\end{itemize} 
\end{proposition}
\proof
Let $\si\in \Si$ be an arbitrary point. We fix a path in $\Si$ from 0
to $\si$. Our goal is to construct flat coordinates in a neighborhood
of $\si$. Given a basis of cycles $\{\alpha_\r\}_{\r\in\mathfrak{R}}$ we denote by
$\alpha_\r^{\si,1}\in H_2(X_{\si,1};\C)$ the parallel transport of
$\alpha_\r$. The polynomial $f(\si,\x)$ is weighted homogeneous, so
there is a natural $\C^*$-action on $\C^3$, s.t., $f(\si,c\cdot\x) =
cf(\si,\x)$ for every $c\in \C^*$. Using this action we define
\ben
\mathcal{A}_\r=\{ (\la z)
\cdot \y \ |\ \la\in (-\infty,0],\quad \y\in\alpha_\r^{\si,1}\}.
\een
Note that $\mathcal{A}_\r$ is a semi-infinite cycle of the type
\eqref{cycle}, so the corresponding oscillatory integral is
convergent. Using the Fubini's theorem we get
\ben
(-2\pi z)^{-3/2}\int_{\mathcal{A}_\r} e^{F(\s,\x)/z}\omega = 
(-2\pi z)^{-3/2}\int_0^{-\infty} e^{\la} \, (\la z)\, 
\int_{\alpha_\r^{\si,1}} 
e^{\sum_{\r\in\mathfrak{R}\setminus{\{\m\}}} s_\r
  \x^\r\la^{(1-d_\r)}z^{-d_\r}}\frac{\omega}{df}\,d\la 
\een
The exponential in the second integral on the RHS is 
\ben
\sum_\kappa 
\frac{\s^\kappa}{\kappa!}\,  
\x^{\kappa}\, \la^{{\rm deg}(\x^\kappa)} \, z^{-d_\kappa},
\een
where the sum is over all sequences $\kappa=(\kappa_\r)_{\r\in\mathfrak{R}\setminus{\{\m\} } }$ of non-negative integers. Substituting the
above expansion we get
\ben
(-2\pi)^{-3/2} z^{-1/2} \sum_\kappa \left(\int_0^{-\infty} e^{\la} \, \la^{1+{\rm
      deg}(\x^\kappa)}\, d\la \right)
\frac{\s^\kappa}{\kappa!}\,  z^{-d_\kappa}\, \int_{\alpha_\r^{\si,1}} 
\x^{\kappa}\,
\frac{\omega}{df}\, .
\een
Comparing with formula \eqref{cek}, we get the following formula for
the oscillatory integral
\beq\label{J-inf}
J_{\mathcal{A}_\r}(\s,z) = 
z^{\frac{1}{2} }\, d\left(
\sum_\kappa\,  z^{-d_\kappa} 
\frac{\s^\kappa}{\kappa!} \,
\sum_{\r\in\mathfrak{R}} \, 
\frac{c_{\r,\kappa}(\si)}{\pi_{A}} \LD \Phi_\r(\si,1),\alpha_\r\RD\right).
\eeq
The oscillatory integrals $J_{\mathcal{A}_\r}(\s,z)$ are solutions to the
differential equations \eqref{frob_eq1} and \eqref{frob_eq2}. On the
other hand near $z=\infty$ these equations have a fundamental
solution of the type $S_\t(z)z^{\Theta}$, where
$S_\t(z)=1+S_{\t,1}z^{-1}+\cdots$ and $S_{\t,k}\in
\operatorname{Hom}_\C(H,H)$. Therefore, we can choose the cycles
$\{\alpha_\r\}$ in such a way that 
\beq\label{JS}
J_{\mathcal{A}_\r}(\s,z) = S_\t(z)\,z^{\Theta}\, dt_\r = z^{\frac{1}{2}-d_\r} \left(
dt_\r + z^{-1}S_{\t,1}(dt_\r) +\cdots \right),
\eeq
where $\t=(t_\r)$ is a flat coordinate system. Note that $d_\r>0$ for
$\r\neq\m$ and $d_\m=0$. Therefore, we have
\ben
J_{\mathcal{A}_\r}(\s,z)  = z^{\frac{1}{2}-d_\r} dt_\r +
z^{\frac{1}{2}-1}\delta_{\r,\m}
S_{\t,1}(dt_\m)+\cdots, 
\een
where the dots stand for terms involving higher order powers of
$z^{-1}$. 
Let us choose the flat
coordinates in such a way that $\d_0= \mathbf{1}$ and $(\d_\m,\d_0)=1$, then we
have
\ben
S_{\t,1} (dt_\m) = S_{\t,1} (\mathbf{1}) =\frac{1}{2} d\,\left( \sum_{\r',\r''\in\mathfrak{R}}
(\d_{\r'},\d_{\r''})\, t_{\r'}t_{\r''} \right).
\een
Finally we get 
\ben
J_{\mathcal{A}_\r}(\s,z)  = z^{\frac{1}{2}}d 
\left( z^{-d_\r}t_\r + z^{-1}\delta_{\r,\m}\,\frac{1}{2}\sum_{\r',\r''\in\mathfrak{R}}
(\d_{\r'},\d_{\r''})\,t_{\r'}t_{\r''}  \right) + \cdots.
\een
All statements in the Proposition follow by comparing the above
formula with \eqref{J-inf}.
\qed 

\subsection{The monodromy of the flat coordinates}
Let us choose the fundamental matrix $\mathfrak{F}$ of the Picard--Fuchs
equations (see \eqref{F-matrix}) in such a way that
$\{\alpha_\r\}_{\r\in \mathfrak{R}}$ and $\{A_\r\}_{\r\in \mathfrak{R}}$ are dual
bases. Furthermore, since $C_{0,\m}\,\alpha_\m$ is a tube cycle we can
find $B\in H_1(E_0;\C)$ such that $L(B):=C_{0,\m}\,\alpha_\m$, so we have
\ben
t_\m = \frac{\pi_B(\si)}{\pi_A(\si)}. 
\een
The flat coordinate $t_0$ is such that $\d/\d t_0=1$. Therefore, the
coefficient in front of $s_0$ in $t_0(\s)$ must be 1. On the other
hand, using formulas \eqref{Cer} and \eqref{cek} we get
\ben
C_{0,\m}(\s) = c_{0,0},\quad C_{0,0}(\s) = c_{0,0}\, s_0 + \cdots,
\een 
where the dots stand for at least quadratic terms in $\s$. It follows
that $L(A) = C_{0,\m}\, \alpha_0.$

Let $\gamma$ be a loop in $\Si$ based at the reference point
$\si=0$. Let us denote by $[\rho_{\neq 0}(\gamma)]_{\r,\r'}$ the matrix
of the linear operator $\rho_{\neq 0}(\gamma)$ in the basis
$\{\alpha_r\}_{\r\in\mathfrak{R}_{\rm tw}}$, i.e., 
\ben
\rho_{\neq 0}(\gamma)(\alpha_{\r})= \sum_{\r'\in \mathfrak{R}_{\rm tw}} [\rho_{\neq
  0}(\gamma)]_{\r',\r}\, \alpha_{\r'}.
\een
Since the monodromy representation $\rho_{\neq 0}$ commutes with the
classical monodromy $J$ and $\{\alpha_\r\}$ is an eigenbasis for $J$,
the matrix $[\rho_{\neq 0}(\gamma)]$ is block diagonal
\ben
[\rho_{\neq 0}(\gamma)]_{\r',\r}\neq 0 \quad\Rightarrow\quad 
d_{\r'}=d_\r.
\een
Similarly let us denote by $[\rho_0(\gamma)]$ the matrix of $\rho_0(\gamma)$ in the
basis $\{\alpha_\m,\alpha_0\}$
\ben
\rho_0(\gamma)(\alpha_\m)& = & [\rho_0(\gamma)]_{\m,\m}\alpha_\m+
[\rho_0(\gamma)]_{0,\m}\alpha_0 ,\\
\rho_0(\gamma)(\alpha_0) & = & [\rho_0(\gamma)]_{\m,0}\alpha_\m+
[\rho_0(\gamma)]_{0,0}\alpha_0 .
\een
The space $\lieh_0\cong H_1(E_0,\C)$ is equipped with a symplectic
form that comes from the intersection pairing. By continuity, the
linear transformation $\rho_0(\gamma)$ is a symplectic transformation,
i.e., the matrix
\ben
\begin{bmatrix}
a & b \\
c & d
\end{bmatrix}:= [\rho_0(\gamma)]^T = 
\begin{bmatrix}
[\rho_0(\gamma)]_{\m,\m} & [\rho_0(\gamma)]_{0,\m} \\
[\rho_0(\gamma)]_{\m,0} & [\rho_0(\gamma)]_{0,0}
\end{bmatrix}\quad \in \operatorname{SL}_2(\C).
\een 
An immediate corollary of Proposition \ref{flat-c} is the
transformation rule for the flat coordinates under the analytic
continuation along $\gamma$.
\begin{corollary}\label{flat-c-mon}
The analytic continuation along a loop $\gamma$ transforms the flat
coordinates as follows
\ben
\widetilde{t}_\m & = &  \frac{a t_\m + b}{c t_\m + d} ,\\
\widetilde{t}_0 & = &  t_0 + \frac{c}{2\,(c t_\m +d)} 
\sum_{\r',\r''\in  \mathfrak{R}_{\rm tw}} (\d_{\r'},\d_{\r''})\, t_{\r'}t_{\r''},\\
\widetilde{t}_\r & = & \frac{1}{ct_\m+d}\sum_{\r'\in \mathfrak{R}_{\rm tw}: d_{\r'}=d_\r} [\rho_{\neq
  0}(\gamma)]_{\r',\r}\, t_{\r'},\quad \r\in \mathfrak{R}_{\rm tw}.  
\een
\end{corollary}

\subsection{Monodromy of the asymptotical operator}

Recall the notation from Section \ref{asymptotic:sec}. Let us identify
the space of linear operators $\operatorname{Hom}_\C(\C^\mu,H)$ with
the space of square matrices of size $\mu$ by fixing the a basis
$\{\phi_\r(\x):=\x^\r\in H; \r\in\mathfrak{R}\}$ and a standard basis $\{e_i\in \C^\mu; 1\leq i\leq \mu\}$, i.e., 
\ben
\mathbb{A}(e_i)=\sum_{\r\in \mathfrak{R}} [\mathbb{A}]_{\r,i} \, \phi_\r.
\een
The asymptotical operator $\Psi_\s\,R_\s\,e^{U_\s/z}$ can be viewed as a matrix with entries
formal asymptotical series. Let us fix a loop $\gamma$ in $\Sigma$. We
would like to find out how the operator changes under the analytic
continuation along $\gamma$. The answer can be stated in the following
way. Let $M(\gamma,t) \in \operatorname{Hom}_\C(H,H)$ be the
operator whose matrix is
\ben
[M(\gamma,t)]_{\r',\r} = \frac{1}{c t_\m+d} \, 
\frac{\d t_{\r'}}{\d\widetilde{t}_\r} -
z\,c\,\delta_{\m,\r}\delta_{0,\r'},
\een
where $c,d,$ and $\widetilde{t}_\r$ are determined via $\gamma$ as it
was stated in Corollary \ref{flat-c-mon}.  Analytic continuation along
$\gamma$ transforms the sequence of critical values
$(u_1(\s),\dots,u_\mu(\s))$ via some permutation $p$. Let us denote by
$P(\gamma)\in  \operatorname{Hom}_\C(\C^\mu,\C^\mu)$ the linear
operator whose matrix is given by
\ben
[P(\gamma)]_{i,j} = \delta_{i,p(j)}.
\een
\begin{proposition}\label{pr:mon-aop}
The analytic continuation along the loop $\gamma$ transforms the
asymptotical operator $\Psi_\s\,R_\s\,e^{U_\s/z}$ into 
\ben
\leftexp{T}{M}(\gamma,t)\, \Psi_\s\,R_\s\,e^{U_\s/z}\, P(\gamma),
\een
where for a linear operator $\mathbb{A}:H\to H$ we denote by $\leftexp{T}{\mathbb{A}}$
the transpose of $\mathbb{A}$ with respect to the residue pairing $(\cdot,\cdot)$. 
\end{proposition}
\proof
Let us denote by $I_i(\s,z)(1\leq i\leq \mu)$ the stationary phase
asymptotic of the oscillatory integral
\ben
(-2\pi z)^{-3/2} \int_{\B_i} e^{F(\s,\x)/z} d^3\x.
\een
By definition the asymptotical operator is defined by the following identity:
\ben
(\phi_\r,\Psi_\s\,R_\s\,e^{U_\s/z}\, e_i) = z\,\frac{\d}{\d t_\r}\left(\frac{I_i(\s,z)}{\pi_A}\right),\quad
\r\in \mathfrak{R},\quad 1\leq i\leq \mu.
\een
The analytic continuation along $\gamma$ transforms the above matrix
into
\beq\label{ei-entry}
 z\,\frac{\d}{\d \widetilde{t}_\r}\left(\frac{I_{p(i)}(\s,z)}{\pi_A\,(ct_\m+d)}\right)
=\sum_{\r'\in \mathfrak{R}} 
z\,\frac{\d}{\d t_{\r'}}
\left(\frac{I_{p(i)}(\s,z)}{\pi_A\,(ct_\m+d)}\right)\frac{\d t_\r'}{\d \widetilde{t}_\r},
\eeq
Note that 
\ben
\sum_{\r'\in \mathfrak{R}}\, z\,\frac{\d}{\d t_{\r'}}\left(\frac{1}{ct_\m+d}\right)\frac{\d t_{\r'}}{\d \widetilde{t}_\r} = z\,\frac{\d t_\m}{\d \widetilde{t}_\m}\,\left(\frac{-c}{(ct_\m+d)^{2}}\right)\delta_{\m,\r} = -z\,c\,\delta_{\m,\r}.
\een
and
\ben
\frac{I_{p(i)}(\s,z)}{\pi_A} = z\,\frac{\d}{\d t_0}\left(\frac{I_{p(i)}(\s,z)}{\pi_A}\right).
\een
Hence the RHS of \eqref{ei-entry} becomes 
\ben
\sum_{\r'\in \mathfrak{R}} \left(\frac{z}{ct_\m+d}\right)\frac{\d t_{\r'}}{\d \widetilde{t}_\r}\frac{\d}{\d t_{r'}}\left(\frac{I_{p(i)}(\s,z)}{\pi_A}\right)
-\left(z\,c\,\delta_{\m,\r} \right)\ z\,\frac{\d}{\d t_0}\left(\frac{I_{p(i)}(\s,z)}{\pi_A}\right).
\een
It remains only to check that the above expression coincides with 
\ben
\left(M(\gamma,t)(\phi_\r), \Psi_\s\, R_\s\, e^{U_\s/z}\,P\, e_i\right) = \sum_{\r'\in\mathfrak{R}}
[M(\gamma,t)]_{\r',\r}\, z\,\frac{\d}{\d t_{\r'}}\left(\frac{I_{p(i)}(\s,z)}{\pi_A}\right). \qed
\een

Let us introduce the linear operators (cf. Corollary \ref{flat-c})
\ben
J(\gamma,t):H\to H,\quad [J(\gamma,t)]_{\r',\r} = \frac {\d t_{\r'}}{\d \widetilde{t}_\r}
\een
and 
\ben
X(\gamma,t):\H\to \H,\quad 
X(\gamma,t) = 1-\left(\frac{c z}{ct_\m+d}\right) \phi_\m\bullet_{\s=0},
\een
where $\phi_\m\bullet_{\s=0}:H\to H$ is the operator of multiplication by $\phi_\m$ in the Jacobi algebra $H$. Note that $X(\gamma,t)$ is a symplectic transformation. 
\begin{proposition}\label{ans:transf}
The analytic continuation along the loop $\gamma$ transforms the total ancestor potential $\A_\s(\hbar;\q)$ into
\ben
\left(\widehat{X}(\gamma,t)\, {\A}_\s\right)\left((ct_\m+d)^2\hbar;J(\gamma,t)\q\right),
\een
where we first apply the operator $ \widehat{X}(\gamma,t)$ and then we rescale $\hbar$ and $\q$.
\end{proposition}
\proof
We may assume that $P(\gamma,t)=1$ because $P$ is a permutation
matrix, so its quantization $\widehat{P}$ will leave the product of
Kontsevich--Witten tau functions invariant. Put $M=M_0+z\,M_1.$ Then we
have 
\ben
\leftexp{T}{M}\,\Psi_\s\, R_\s\, e^{U_\s/z} = \widetilde{\Psi}_\s\, \widetilde{R}_\s\, e^{U_\s/z},\quad
\mbox{where}\quad \widetilde{\Psi}_\s=M_0^{-1}\,\Psi_\s,\quad \widetilde{R}_\s=\Psi_\s^{-1}M_0\,\leftexp{T}{M}\,\Psi_\s\, R_\s.
\een
The quantization is in general only a projective representation. However, the quantization of the operators $\Psi_\s^{-1}\,M_0\,\leftexp{T}{M}\,\Psi_\s$ and $R_\s$ involves quantizing only $p^2$ and $p\,q$-terms. Since the cocycle \eqref{cocycle} on such terms vanishes we get
\ben
\left(\widetilde{R}_\s\right)\sphat = \left(\Psi_\s^{-1}\,M_0\,\leftexp{T}{M}\,\Psi_\s\right)\sphat\
\ \widehat{R}_\s.
\een
The operators $M_0$ and $\Psi_\s$ are independent of $z$ and their quantizations by definition are just changes of variables. Hence
\ben
\left(\widetilde{\Psi}_\s\,\widetilde{R}_\s\right)\sphat = \widehat{M_0}^{-1} \left(M_0\,\leftexp{T}{M}\right)\sphat \ \left(\Psi_\s\,R_\s\right)\sphat\ .
\een
By definition $\Delta_i^{-1}$ is $(\d_{u_i},\d_{u_i})$, which gains a
factor of $(ct_\m+d)^{-2}$ under analytic continuation. The ancestor
potential \eqref{ancestor} is transformed into 
\beq\label{ancestor:mon1}
\widehat{M_0}^{-1} \left(M_0\,\leftexp{T}{M}\right)\sphat\ \left( \A_\s\left((ct_\m+d)^2\hbar ;(ct_\m+d) \q\right)\right).
\eeq
Note that 
\ben
M_0^{-1} = (ct_\m+d)\ J(\gamma,t)^{-1},\quad
M_0\,(\leftexp{T}{M})= X(\gamma,t).
\een
It remains only to  notice that the rescaling
\ben
(\hbar,\q)\mapsto \left((ct_\m+d)^2\hbar,(ct_\m+d)\q\right)
\een
commutes with the action of any quantized operator.
\qed

\subsection{Quasi-modularity}
Let us denote by $'t=(t_{\r})_{\r\in \mathfrak{R}\setminus{ \{\m\} } }$ the non-marginal
flat coordinates. It is known that the ancestor potential $\A_\s$
depends analytically on $\s\in D$, where $D\subset \S$ is any open
domain in which the primitive form is single-valued. In particular for
each fixed $\si\in \Si$, we can take the limit 
\ben
\A_\si(\hbar;\q):=\lim_{'t\to 0} \ \A_\s(\hbar;\q).
\een
For the proof of the above statement see \cite{MR}, or more generally
\cite{M}. 

The marginal flat coordinate can be written as 
\ben
t_\m = \frac{a'\tau + b'}{c'\tau+d'},\quad \tau\in \mathbb{H},\quad 
\begin{bmatrix}
a' & b'\\
c' & d'
\end{bmatrix}\in \operatorname{SL}_2(\C).
\een
We define 
\ben
\overline{t}_\m:= \frac{a'\overline{\tau} + b'}{c'\overline{\tau}+d'},
\een
where $\overline{\phantom{\tau} }$ is the standard conjugation in the upper
half-plane $\mathbb{H}$. 
Since the analytic continuation transforms $\tau$ and
$\overline{\tau}$ via the same fractional
linear transformation, we get that  the
analytic transformation along $\gamma$ transforms $t_m$ and
$\overline{t}_\m$ respectively into 
\ben
t_\m\mapsto \frac{a t_\m + b}{c t_\m+d}\quad \mbox{and}\quad 
\overline{t}_\m\mapsto \frac{a \overline{t}_\m + b}{c \overline{t}_\m+d}.
\een
A direct computation shows that 
\ben
-\frac{1}{t_\m-\overline{t}_\m}\mapsto 
-\frac{(c t_\m+d)^2}{t_\m-\overline{t}_\m}+ c(ct_\m+d).
\een
Following \cite{MR}, we define {\em anti-holomorphic} completion of
the ancestor potential
\ben
\widetilde{\A}_\si(\hbar;\q) = (\widetilde{X}(\si,z))\sphat \ \A_\si(\hbar;\q),
\een
where
\ben
\widetilde{X}(\si,z) = 1-\left(\frac{z}{t_\m-\overline{t}_\m}\right)\, \phi_\m\bullet_{\s=0}.
\een
Proposition \ref{ans:transf} yields the following corollary  (cf. \cite{MR}).
\begin{corollary}\label{q-mod}
The analytic continuation along $\gamma$ transforms the modified total
ancestor potential as follows. 
\ben
\widetilde{\A}_\si(\hbar;\q)\mapsto 
\widetilde{\A}_\si\left((ct_\m+d)^2\hbar;J(\gamma,\si)\q\right).
\een
\end{corollary}

\noindent{\em Proof of Theorem \ref{t1}}. 
Assume that $\gamma\in \Gamma(W)$, then $\rho_{\neq 0}(\gamma)=J^n$
for some integer $n$, where $J$ is the classical monodromy operator. Since 
$J(\alpha_\r) = e^{2\pi\sqrt{-1} {\rm deg(\x^{\r})}}\alpha_\r$, we get that
the matrix of $J(\gamma,\si)$ is diagonal and we have
\ben
[J(\gamma,\si)]_{\r,\r} =  
(ct_\m+d)\,e^{2\pi\sqrt{-1}\, n\, {\rm deg(\x^\r)}},\quad \r\in \mathfrak{R}_{\rm tw},
\een
and 
\ben
[J(\gamma,\si)]_{\m,\m} =   (ct_\m+d)^2, \quad
[J(\gamma,\si)]_{0,0} =  1.
\een
The total ancestor potential is invariant under the rescaling
$q_k^\r\mapsto e^{2\pi\sqrt{-1}\, n\, {\rm deg(\x^\r)}} q_k^\r$, because
it is quasi-homegeneous. Hence we may assume that $n=0$. The statement
of the theorem follows from Corollary \ref{q-mod} by comparing the
coefficients of the monomials in $\q$ and $\hbar$. Each
coefficient $\widetilde{c}(t_\m)$ of the modified potential has the
form
\ben
\widetilde{c}(t_\m) = \sum_{i=0}^\infty \frac{c_i(t_\m)}{ (t_\m-\overline{t}_\m)^{i}},
\een 
where $c_i=0$ for $i\gg 0.$ The analytical continuation of
$\widetilde{c}(t_\m)$ is
\ben
\widetilde{c}\, \Big( \frac{a t_\m+b}{c t_\m+d}\Big)= (c t_\m+d)^w \, \widetilde{c}(t_\m), 
\een 
where the identity follows from Corollary \ref{q-mod} and 
$w$ is a non-negative integer depending on the monomial (see \cite{MR}
for more details).
\qed

\section{The modular groups $\Gamma(W)$}
The computation of $\Gamma(W)$ amounts to computing the monodromy
group of several hypergeometric equations of the type 
\beq\label{HG:1}
x(1-x)y''(x)+\left(\gamma-(1+\alpha+\beta)x\right)y'(x)-\alpha\beta\, y(x) =0,
\eeq
where $\alpha,\beta$, and $\gamma$ are positive rational numbers. Let
us begin by briefly reviewing the main steps in the computation.
\subsection{Monodromy of the hypergeometric equations}\label{sec:local}
There are two cases which are used in our work.
\subsubsection{The resonance case} We assume that
$\gamma=\alpha+\beta=1-1/l$, where $l$ is a positive integer. 
Near $x=0$ the hypergeometric equation \eqref{HG:1} admits the following basis of
solutions:
\beq
\left\{
\begin{aligned}
F_1^{(0)}(x) & = 
\frac{\Gamma(\alpha)\Gamma(\beta)}{\Gamma(\alpha+\beta)} \ 
\leftbase{2}{F}_1(\alpha,\beta;\gamma;x), \\
F_2^{(0)}(x) & = 
\frac{\Gamma(1-\alpha)\Gamma(1-\beta)}{\Gamma(2-\alpha-\beta)} \ 
\leftbase{2}{F}_1(1-\alpha,1-\beta;2-\gamma;x)\, x^{1-\alpha-\beta}.
\end{aligned}
\right.
\eeq
Near $x=1$ a basis of solutions is given by 
\beq\label{basis:1}
\left\{
\begin{aligned}
F_1^{(1)}(x) & = 
\leftbase{2}{F}_1(\alpha,\beta;1;1-x), \\
F_2^{(1)}(x) & = 
\leftbase{2}{F}_1(\alpha,\beta;1;1-x)\, \ln(1-x) + \sum_{n=1}^\infty b_n(1-x)^n,
\end{aligned}
\right.
\eeq
where
\ben
b_n=\frac{(\alpha)_n(\beta)_n}{(n!)^2}\,
\Big( \frac{1}{\alpha}+\cdots+\frac{1}{\alpha+n-1}+
\frac{1}{\beta}+\cdots+\frac{1}{\beta+n-1} - 
2\Big(\frac{1}{1}+\cdots+\frac{1}{n}\Big)\Big).
\een
Let us denote by $F^{(a)}(x)$ the column vector with entries $F_1^{(a)}(x)$
and $F_2^{(a)}(x)$ for $a=0,1$; then the local monodromy around $x=0$ acts as
\ben
F^{(0)}(x)\mapsto M_0^T\, F^{(0)}(x);\quad 
M_0:= 
\begin{bmatrix}
1 & 0\\
0 & e^{2\pi i(1-\alpha-\beta)}
\end{bmatrix}.
\een
It follows that the local monodromy around by $x=1$ is given by
\ben
F^{(1)}(x)\mapsto M_1^T\, F^{(1)}(x);\quad  
M_1:= 
\begin{bmatrix}
1 & 2\pi i\\
0 & 1
\end{bmatrix}.
\een
Let $\psi(z)=\Gamma'(z)/\Gamma(z)$ be the digamma function. Put
\beq\label{digamma}
K_1=2\psi(1)-\psi(\alpha)-\psi(\beta),\quad 
K_2=2\psi(1)-\psi(1-\alpha)-\psi(1-\beta).
\eeq
The key to the monodromy computation is the following lemma (see \cite{AS}).
\begin{lemma}\label{conn:matrix}
The series $F^{(0)}(x)$ and $F^{(1)}(x)$ are convergent
in the region $\{x\in \C\ :\ |x|<1, |1-x|<1\}$ and the following formula holds:
\ben
F^{(0)}(x)\mapsto C^{01} F^{(1)}(x),\quad
C^{01}:=
\begin{bmatrix}
K_1 & -1\\
K_2 & -1
\end{bmatrix}
\een
provided the branch of $F^{(1)}(x)$ near $x=1$ is chosen appropriately. 
\end{lemma}
\begin{proof}
The first row of $C^{01}$ directly follows from the formula 15.3.10 in \cite{AS}.
The second row follows from the first row and the hypergeometric identity
\ben
\frac{\Gamma(\alpha-\gamma+1)\Gamma(\beta-\gamma+1)\Gamma(\gamma-1)}{\Gamma(\alpha)\Gamma(\beta)\Gamma(1-\gamma)}\ _2F_1\left(\alpha-\gamma+1,\beta-\gamma+1;2-\gamma;x\right)\,x^{1-\gamma}
\\
=-\ _2F_1\left(\alpha,\beta;\gamma;x\right)
+
\frac{\Gamma(\alpha-\gamma+1)\Gamma(\beta-\gamma+1)}{\Gamma(1-\gamma)\Gamma(\alpha+\beta-\gamma+1)}\ _2F_1\left(\alpha,\beta;\alpha+\beta-\gamma+1;1-x\right).
\een
\end{proof}
The local monodromies $M_0$, $M_1$, and the connection matrix $C^{01}$
completely determine the monodromy representation. 
\begin{comment}
Let us also list a
basis of solutions near $x=\infty$. If $\alpha-\beta\notin\Z$, then  
\ben
\left\{
\begin{aligned}
F_1^{(\infty)} & = 
x^{-\alpha}\ _2F_1\left(\alpha,\alpha-\gamma+1;\alpha-\beta+1;x^{-1}\right),
\\
F_1^{(\infty)} & = 
x^{-\beta}\
_2F_1\left(\beta,\beta-\gamma+1;\beta-\alpha+1;x^{-1}\right).
\end{aligned}
\right.
\een
If $\alpha=\beta$, then
\ben
\left\{
\begin{aligned}
F_1^{(\infty)} & = x^{-\alpha}\
_2F_1\left(\alpha,\alpha-\gamma+1;1;x^{-1}\right) ,\\
F_2^{(\infty)} & = G_{2,2}^{2,2} \Big( \   x^{-1}\ \Big\vert\    {
  \gamma,1 \atop \alpha,\beta} \ \Big),
\end{aligned}
\right.
\een
where $G_{2,2}^{2,2}$ is Meijer's $G$-function.
\end{comment}

\subsubsection{The non-resonance case} 
Now we assume that none of the exponents
\ben
\la_0=1-\gamma,\quad \la_1=\gamma-\alpha-\beta,\quad
\la_\infty=\beta-\alpha
\een
is an integer. Then we fix the following solutions.
Near $x=0$:
\ben
\left\{
\begin{aligned}
F_1^{(0)}(x)& =\ _2F_1\left(\alpha,\beta;\gamma;x\right), \\
F_2^{(0)}(x)& =\
_2F_1\left(\alpha-\gamma+1,\beta-\gamma+1;2-\gamma;x\right)\,
x^{1-\gamma}\ .
\end{aligned}
\right.
\een
Near $x=1$:
\beq\label{eq:unity-not0}
\left\{
\begin{aligned}
F_1^{(1)}(x) & =  \
_2F_1\left(\alpha,\beta;\alpha+\beta-\gamma+1;1-x\right), \\
F_2^{(1)}(x) & = 
\
_2F_1\left(\gamma-\alpha,\gamma-\beta;\gamma-\alpha-\beta+1;1-x\right)\,
(1-x)^{\gamma-\alpha-\beta}.
\end{aligned}
\right.
\eeq
\begin{comment}
Near $x=\infty$:
\beq\label{eq:infinity-not0}
\left\{
\begin{aligned}
F_1^{(\infty)} & = 
\
_2F_1\left(\alpha,\alpha-\gamma+1;\alpha-\beta+1;x^{-1}\right)\,
x^{-\alpha} \, ,\\
F_2^{(\infty)} & = 
\
_2F_1\left(\beta,\beta-\gamma+1;\beta-\alpha+1;x^{-1}\right)\,
x^{-\beta}\, .
\end{aligned}
\right.
\eeq
\end{comment}
Let us denote by $F^{(a)}(x)(a=0,1)$ the corresponding column vectors. Then the key fact is the following. Just like in the previous
case, there is a common contractible domain where both $F^{(0)}$ and
$F^{(1)}(x)$ are convergent and hence one can define a connection matrix
$D^{01}$, such that $F^{(0)}(x)\mapsto D^{01}\, F^{(1)}(x)$. 
\begin{lemma}\label{ac:non-res}
The connection matrices are given by the following formulas:
\beq\label{connect-twist}
D^{01}=
\left[\ 
\begin{aligned}
&\frac{\Gamma(\gamma)\Gamma(\gamma-\alpha-\beta)}{\Gamma(\gamma-\alpha)\Gamma(\gamma-\beta)}
& &
\frac{\Gamma(\gamma)\Gamma(\alpha+\beta-\gamma)}{\Gamma(\alpha)\Gamma(\beta)}
\\
&
\frac{\Gamma(2-\gamma)\Gamma(\gamma-\alpha-\beta)}{\Gamma(1-\alpha)\Gamma(1-\beta)}
& &
\frac{\Gamma(2-\gamma)\Gamma(\alpha+\beta-\gamma)}{\Gamma(1+\alpha-\gamma)\Gamma(1+\beta-\gamma)}
\end{aligned}\ 
\right]\, .
\eeq
\end{lemma}

\subsection{The group $\widetilde{\Gamma}(W)$}\label{PF:mon}
Recall that $\widetilde{\Gamma}(W)=\operatorname{Im}(\rho_0)$, where
\ben
\rho_0:\pi_1(\Si)\to \operatorname{GL}(\lieh_0),\quad \lieh_0=H_2(X_{0,1};\C)^J\cong H_1(E_0;\C),
\een 
is the $J$-invariant part of the monodromy representation. Given a
basis $\{A,B\}$ of $\lieh_0$, we can compute the monodromy action on
that basis by computing the monodromy under analytic continuation of
the corresponding period integrals $\pi_A(\si)$ and $\pi_B(\si)$ (see
Section \ref{sec:pf}). Let $D=\si \d_\si$, then both $\pi_A(\si)$ and $\pi_B(\si)$ satisfy a second order differential
equation
\beq\label{PF:1}
D(D-1)\pi(\si)-C\, \sigma^l(D+l\alpha)(D+l\beta)\pi(\si)=0,
\eeq
where $C$ is some constant, $l\in\mathbb{Z}$, $\alpha,\beta\in\mathbb{Q}$ and $\alpha+\beta=1-\frac{1}{l}.$ 

Using the substitution $x=C\,\sigma^l$, the differential equation \eqref{PF:1} becomes the standard hypergeometric equation
\eqref{HG:1} (with $\gamma=\alpha+\beta$), so the monodromy can be
computed as explained above. We
fix a reference point on $\C-\{0,1,\infty\}$ near $x=1$ and a basis of
solutions as \eqref{basis:1}. Denote by $M_0^{HG}$ and $M_1^{HG}$ the monodromy transformations of the column vector $F^{(1)}(x)$ corresponding to paths
going around $x=0$ and $x=1$, i.e. $F^{(1)}(x)\mapsto M_a^{HG}\,F^{(1)}(x)$, $a=0,1$. We may choose the paths that 
\beq\label{HG-la0}
M_1^{HG}=M_1^T; \quad M_0^{HG}= (C^{01})^{-1}\,M_0^T\,C^{01},
\eeq
where $M_0,M_1$, and $C^{01}$ are respectively the local monodromies
near $x=0,1$ and the matrix giving the analytic continuation from
$x=1$ to $x=0$ (see Lemma \ref{conn:matrix}). 

Our substitution $x=C\, \sigma^l$ is a covering $\Si-\{0\}\to \C-\{0,1,\infty\}$ of
degree $l$. Let us denote by $M_i^{PF},$ $1\leq
i\leq l$, the monodromy transformations corresponding to loops going
around $\si=p_i,$ which are the singularities of the differential equation \eqref{PF:1}, 
\beq\label{sing-pts}
p_i=C^{-1/l}\eta^i, \quad \eta=\exp(2\pi\sqrt{-1}/l).
\eeq
Since $\Sigma=\C\backslash \{p_1,\dots,p_l\},$ 
$\widetilde{\Gamma}(W)$ is generated by $M_i^{PF},$ $1\leq i\leq l.$
Lifting the reference point and choosing the 
loops appropriately, we can arrange that 
\beq\label{PF-la=0}
M_i^{PF} = (M_0^{HG})^{i-1}\, M_1^{HG}\,
(M_0^{HG})^{1-i},\quad 1\leq i\leq l.
\eeq 
It remains only to explain how to find a basis of solutions
$\{F_1^{\rm GW},F_2^{\rm GW}\}$ that corresponds to periods of the
elliptic curve, i.e., 
\ben
F_1^{\rm GW}(\si)=\pi_A(\si),\quad F_2^{\rm GW}(\si)=\pi_B(\si),
\een
where $\{A,B\}$ is an integral basis of $H_1(E_0;\Z)$, s.t. $A\circ B=1.$ The
$j$-invariant of $E_\si$ has the form
\beq\label{j-inv}
j(\si) = \frac{P(\si)}{(1-C\, \si^l)^N},\quad P(\si)\in \C[\si],
\eeq
where the zeroes of the polynomial in the denominator are precisely
the singular points in \eqref{sing-pts}. 
Since we are interested in the matrices of the monodromy transformations, we have the freedom to rescale the
above basis by any non-zero constant, so we may assume that 
\beq\label{GW-la0}
\begin{bmatrix}
F_1^{GW}\\
F_2^{GW}
\end{bmatrix}
= 
K\,
\begin{bmatrix}
F_1^{(1)}\\
F_2^{(1)}
\end{bmatrix},\quad
K=\begin{bmatrix}
1  &   0\\
-\frac{a}{b} &  \frac{1}{b}
\end{bmatrix},
\eeq  
where $F_1^{(1)}, F_2^{(2)}$ are viewed as multi-valued functions of $\sigma$ via the substitution $x=C\,\sigma^l$, 
and $a$ and $b$ are some non-zero constants.
\begin{lemma}\label{sol:GW}
The solutions $F_i^{\rm GW}$, $i=1,2,$ correspond to periods of the elliptic curve (normalized as above) if and
only if
\ben
a=\frac{1}{N}\, (\ln P(p_1)+2\pi\sqrt{-1} m),\quad b=2\pi\sqrt{-1}/N,
\een 
where $p_1=C^{-1/l}$ and $m$ is some integer. 
\end{lemma}
\proof
Since 
$\tau:=F_2^{GW}/F_1^{GW}$
is the modulus of the elliptic curve, we
must choose $a$ and $b$ in such a way that $j(\si)=1/q+\cdots,\
q=e^{2\pi i\tau}.$ Inverting the relation \eqref{j-inv} near $\si=p_1$,
  we find:
\ben
p_1-\si = \frac{e^a}{C\prod_{j=2}^l (p_1-p_j)}\, e^{b\tau}\Big(1+O(e^{b\tau})\,\Big).
\een
Then the lemma follows from
\ben
j(\si) = \frac{P(p_1)}{e^{N(b\tau+a)}}+\cdots\quad
\qed
\een
Let us point out that the value of $a$ is fixed only up to $2m\pi\sqrt{-1}/N$ for some $m\in\Z$.  This corresponds to the fact that while there is a unique choice of an invariant
cycle $A\in H_1(E_\si;\Z)$ near $\si=p_1$, so that  up to a constant  $\pi_A(\si)$ agrees
with  $F_1^{GW}$, for the second cycle $B\in H_1(E_\si;\Z)$, we have
the freedom to add any integer multiple of $A$.
According to \eqref{HG-la0}, \eqref{PF-la=0}, and \eqref{GW-la0} we have
\begin{lemma}
Let $F^{GW}$ be the column vector with entries $F^{GW}_1$, $F^{GW}_2$. The monodromy transformation along the fixed loop going around $\sigma=p_i$ acts on $F^{GW}$ by
\beq\label{modular-inv}
F^{GW}\mapsto M_{i,0}^T\, F^{GW}; \quad M_{i,0}^{T}=K\,M_i^{PF}\,K^{-1}.
\eeq
As a consequnce, the group $\widetilde{\Gamma}(W)$ is isomorphic to the subgroup of $\operatorname{SL}_2(\Z)$  
generated by $M_{i,0}$, $i=1,\dots, l.$
\end{lemma} 
Now we are in a position to prove Theorem \ref{t2}. Some of the computations are quite cumbersome, so
it is better to use some computer softawre, e.g. Mathematica or Maple. 
\subsection{The Fermat $E_6^{(1,1)}$ case}
The polynomial in this case is 
\ben
f(\si,\x) = x_1^3+x_2^3+x_3^3+\si\,x_1x_2x_3.
\een
Its modular group is well known. For the sake of completeness, let us recall the computation.
The $j$-invariant of this family is 
\ben
j(\si) = -\frac{\si^3(\si^3-216)^3}{(\si^3+27)^3}.
\een
The weights of the equation \eqref{PF:1} are $\alpha=\beta=1/3$,
$\gamma=2/3.$
According to \eqref{modular-inv}, $\widetilde{\Gamma}(W)=\Gamma(3)$ since for three singular points $p_i=-3e^{2\pi\sqrt{-1}i/3}, i=1,2,3$, we have
\ben
M_{1,0}=\begin{bmatrix} 1 & 3\\ 0 & 1\end{bmatrix},\quad
M_{2,0}=\begin{bmatrix} 1-3m & 3m^2\\ -3 & 1+3m\end{bmatrix},\quad
M_{3,0}=\begin{bmatrix} 4-3m & 3(m-1)^2\\ -3 & 3m-2\end{bmatrix},
\een
where the choice of $m$ depends on the coice of a symplectic basis
$\{A,B\}$ in $H_1(E_\si;\Z)$. 
For the period integrals $\Phi_{\r}(\sigma)$, $\r\in\mathfrak{R}_{\rm tw},$ the monodromy acts on the corresponding flat sections by multiplication of $e^{2\pi\sqrt{-1}\deg\phi_{\r}}$, which is just $J$ or $J^2$, with $J$ the classical monodromy operator. 
Thus $\Gamma(W)=\widetilde{\Gamma}(W)=\Gamma(3)$.

\subsection{The Fermat $E_7^{(1,1)}$ case}
The polynomial in this case is 
\ben
f(\si,\x) = x_1^4+x_2^4+x_3^2+\si\,x_1^2x_2^2.
\een
The $j$-invariant is 
\ben
j(\si) = 16\,\frac{(\si^2+12)^3}{(4-\si^2)^2}.
\een
Thus $C=1/4$, $l=2$ and $P(\si)=(\sigma^2+12)^3.$
\subsubsection{Monodromy of the invariant part}
The parameters of the equation \eqref{PF:1} are $\alpha=\beta=1/4$,
$\gamma=1/2$. 
We consider two singular points $p_1=2$ and $p_2=-2$.
According to \eqref{modular-inv},  the monodromy group is generated by the following two
matrices:
\ben
M_{1,0}=\begin{bmatrix} 1 & 2\\ 0 & 1\end{bmatrix},\quad
M_{2,0}=\begin{bmatrix} 1-2m & 2m^2\\ -2 & 1+2m\end{bmatrix}\, ,
\een
where the choice of $m$ depends on the coice of a symplectic basis
$\{A,B\}$ in $H_1(E_\si;\Z)$. 
\begin{comment}
$$a=(ln(16^3)+2\pi\sqrt{-1}m)/2=6ln2+m\pi\sqrt{-1}$$
\end{comment}
We choose the basis to be such that
$m=0$; then the monodromy group $\widetilde{\Gamma}(W)$ is the group of all matrices 
\ben
g=\begin{bmatrix} a & b\\ c & d\end{bmatrix}\in {\rm  SL}_2(\Z), \quad
a\equiv d\equiv1\ ({\rm mod}\, 4),\quad
b\equiv c\equiv0\ ({\rm mod}\, 2). 
\een
This is an index 2 subgroup of $\Gamma(2)$. 

\subsubsection{Monodromy of the twisted sector}
We fix the following basis of the twisted sectors in the Jacobian algebra $\mathscr{Q}_W$:
\ben
\phi_{\r}(\x)=x_1^{r_1}x_2^{r_2}x_3^{r_3},\quad \r=\{(100),(010),(200),(110),(020),(210),(120)\}.
\een
According to \cite{MS}, the corresponding period integrals
$\Phi_\r(\si)$ satisfy
\beq\label{de:1st-order}
(4-\si^2) \d_\si\,\Phi_\r(\si) =2\, {\rm deg}(\phi_\r)\, \si\
\Phi_\r(\si),\quad \r\neq (200),(020), 
\eeq
and
\beq
\left\{
\begin{aligned}\label{de:2nd-order}
(4-\si^2)\, \d_\si\Phi_{200}(\si) & =  \frac{\si}{2}\,
\Phi_{200}(\si)-\Phi_{020}(\si)\\
(4-\si^2)\, \d_\si\Phi_{020}(\si) & =  -\Phi_{200}(\si)+\frac{\si}{2}\,
\Phi_{020}(\si)
\end{aligned}
\right.
\eeq
The monodromy of the equations \eqref{de:1st-order} around $\si=p_i$ is straightforward
to compute:
\ben
\Phi_\r\mapsto \ e^{2\pi\sqrt{-1}\,  {\rm deg}(\phi_\r)}  \Phi_\r\, .
\quad i=1,2.
\een
For the system \eqref{de:2nd-order}, we can first obtain a second order
differential equation for $\Phi_{200}$, which after the substitution
$x=\si^2/4$ becomes the hypergeometric equation with weights $\alpha=3/4,\beta=1/4,\gamma=1/2.$ It follows that the system can be
solved as follows: 
\beq\label{E7:fund-solution}
\begin{bmatrix} \Phi_{200} \\ \Phi_{020}\end{bmatrix} = 
\begin{bmatrix} 
   F_1^{(1)}  &    F_2^{(1)}\\
L F_1^{(1)}  &  L F_2^{(1)}
\end{bmatrix}
\,  \begin{bmatrix} A_{100} \\ A_{020}\end{bmatrix} 
\eeq
where $A_{200}$ and $A_{020}$ are flat sections of the vanishing cohomology
bundle and 
\ben
L=-(4-\si^2)\d_\si +\si/2.
\een
Let us denote by $M_{i, 1/4}$ the $2\times 2$ matrix defined by the analytic continuation along 
a simple loop $\gamma_i$ around the point $\si=p_i$  of the following vector-valued function:
\ben
F^{(1)}\mapsto M^T_{i,1/4}\, F^{(1)},\quad F^{(1)}=(F_1^{(1)},F_2^{(1)})^T,
\een
where the reason to use the transposition operation ${ }^T$ will become clear shortly.  
Since the LHS of \eqref{E7:fund-solution} consists of holomorphic sections of the
vanishing cohomology bundle, it must be invariant under analytic
continuation along any loop in $\Si$. It follows that the monodromy of the flat
sections is
\ben
\begin{bmatrix} A_{100} \\ A_{020}\end{bmatrix}  \mapsto
\left(M_{i,1/4}\right)^{-1}\, 
\begin{bmatrix} A_{100} \\ A_{020}\end{bmatrix} .
\een
Note that in the basis $\{\alpha_{100},\alpha_{020}\}$ dual to $\{A_{100},A_{020}\}$ the matrix of the 
monodromy transformation $\rho(\gamma_i)$ is precisely $M_{i,1/4}$. 
On the other hand, arguing as in Section \ref{PF:mon}, we get
\ben
M_{i,1/4} = \Big( (M_0^{HG})^{i-1}\, M_1^{HG} \, (M_0^{HG})^{(1-i)}\Big)^T,\quad i=1,2,
\een
where $M_0^{HG}$ and $M_1^{HG}$ are the monodromies of the corresponding
hypergeometric equation around $x=0$ and $x=1$ respectively. Let us
denote by 
\ben
D_0=
\begin{bmatrix} 1 & 0 \\ 0 & e^{2\pi\sqrt{-1}(1-\gamma)}\end{bmatrix},\quad 
D_1 = 
\begin{bmatrix} 1 & 0 \\0 &  e^{2\pi\sqrt{-1}(\gamma-\alpha-\beta)}\end{bmatrix}
\een
the local monodromies of the hypergeometric equation; then (see Lemma \ref{ac:non-res})
\ben
M_1^{HG}=D_1^T; \quad M_0^{HG}=(C^{01})^{-1}\, D_0^T\, C^{01}.
\een
A straightforward computation yields
\ben
M_{1,1/4} = 
\begin{bmatrix} 1 & 0 \\ 0 & -1\end{bmatrix} ,\quad
M_{2,1/4} = 
\begin{bmatrix} -1 & 0 \\ 0 & 1\end{bmatrix}\, .
\een
The group generated by these two matrices is $\Z/2\times \Z/2$. Let us
point out that this agrees with the so-called Schwarz list (see
\cite{Ma}), since the exponents of the hypergeometric equation are
\ben
\la_0=1-\gamma=\frac{1}{2},\quad \la_1=\gamma-\alpha-\beta =
-\frac{1}{2},\quad
\lambda_\infty = \beta-\alpha = -\frac{1}{2}.
\een
\begin{lemma}\label{mod-group:x9}
The kernel ${\rm Ker}\ (\rho_0)=0$ and the modular group $\Gamma(W)=\Gamma(4).$
\end{lemma}
\proof
Let $\widetilde{\Si}\to \Si$ be the universal cover of $\Si$ and let
$\tau:\widetilde{\Si}\to\mathbb{H}$ be the map given by the quotient of
two periods: $\tau=\pi_B/\pi_A$. Using the Picard-Fuchs equation we
find that
\ben
d\tau = \frac{{\rm Wr}(\pi_A,\pi_B)}{\pi_A^2}\, d\si \neq 0.
\een 
By the implicit function theorem $\tau$ is a local homeomorphism. 
Using the $j$-invariant we see that the map $\tau$ is finite and
surjective. Finite maps between analytic varieties are proper. It
follows that $\tau$ is a covering. Finally, since both $\widetilde{\Si}$ and
$\mathbb{H}$ are simply connected, $\tau$ must be an isomorphism. 

The monodromy group $\widetilde{\Gamma}(W)$ induces an action on
$\mathbb{H}.$ If we assume that ${\rm Ker}\ (\rho_0)$ is non-trivial; then we can find two
different points on $\widetilde{\Si}$ such that their $\tau$-images
coincide -- contradiction. In particular, the condition
\eqref{kernels} is satisfied and the map \eqref{rho-W} is well
defined. Moreover, using our computation of the monodromy of the
twisted periods, we get that $\operatorname{Im}(\rho_W)\cong
\Z/2\times \Z/2$, i.e., we have a surjective group homomorphism 
\ben
\rho_W: \widetilde{\Gamma}(W)\longrightarrow \Z/2\times \Z/2,
\een
which implies that the index of ${\rm Ker} \ \rho_W$ in $
\widetilde{\Gamma}(W)$ is 4. On the other hand, it is easy to see
that $\Gamma(4)$ can be generated by the following (see \cite{Mu}) 6
elements
\ben
\begin{bmatrix} 1 & -4\\0 & 1\end{bmatrix},\quad
\begin{bmatrix} -3 & -4\\4 & 5\end{bmatrix},\quad
\begin{bmatrix} 1 & 0\\4 & 1\end{bmatrix},\quad
\begin{bmatrix} 9 & -4\\16 & -7\end{bmatrix},\quad
\begin{bmatrix} 5 & -4\\4 & -3\end{bmatrix},\quad
\begin{bmatrix} 9 & -16\\4 & -7\end{bmatrix},\quad
\een
We leave it to the reader to check that each of these matrices is in
the kernel of $\rho_W$. For example:
\ben
\begin{bmatrix} -3 & -4\\4 & 5\end{bmatrix}=
\left(\,(M_{2,0})^{-1}\, M_{1,0}\,\right)^2.
\een
Under $\rho_W$ the matrix $M_{i,0}$ is mapped to $M_{i,1/4}$. Using
the explicit formulas for $M_{i,1/4}$ we get that the
above matrix is in the kernel of $\rho_W$. The remaining $5$ matrices
are treated in a similar fashion. Hence we have 
\ben
\Gamma(4)\subseteq {\rm Ker} \ \rho_W \subset
\widetilde{\Gamma}(W)\subset \Gamma(2)\subset {\rm SL}_2(\Z).
\een
Since the index of $\Gamma(2)$ in $\operatorname{SL}_2(\Z)$ is $6$ (see \cite{Mi}), the index of  
${\rm Ker}\, (\rho_W)$ in $ {\rm SL}_2(\Z)$ must be $4\times 2\times 6=48$. However, the index of
$\Gamma(4)$ in $\operatorname{SL}_2(\Z)$ is also $48$. The Lemma follows.
\qed

\subsection{The Fermat $E_8^{(1,1)}$ case}
The polynomial in this case is
 \ben
f(\si,\x)=x_1^6+x_2^3+x_3^2+\si\, x_1^4 x_2.
\een
The $j$-invariant is 
\ben
j(\si)=1728\,\frac{4\si^3}{4\si^3+27}\, .
\een
Thus $C=-\frac{4}{27}$,
$l=\frac{1}{3}$ and $P(\si)=256\si^3$. 
\subsubsection{Monodromy of the invariant part}
The parameters of the hypergeometric equation \eqref{HG:1} are
$\alpha=\frac{1}{12},\beta=\frac{7}{12},\gamma=\frac{2}{3}$.
\begin{comment}
$$a=ln(-12^3)+2\pi\sqrt{-1}m=3ln3+4ln2+(2m+1)\pi\sqrt{-1}$$
\end{comment}
The singular points are 
$$p_i=-3\cdot 4^{-1/3}\, e^{2\pi
  \sqrt{-1}\,(i-1)/3}, \quad 1\leq  i\leq 3.$$ 
According to \eqref{modular-inv}, the corresponding monodromy
transformations are 
\ben
M_{1,0}=
\begin{bmatrix} 1 & 1 \\ 0 & 1\end{bmatrix},\quad
M_{2,0}=
\begin{bmatrix} -m & (1+m)^2 \\ -1 & 2+m\end{bmatrix},\quad
M_{3,0}=
\begin{bmatrix} 1-m & m^2 \\ -1 & 1+m\end{bmatrix},\quad
\een
where the integer $m$ depends on the choice of a symplectic basis in
$H_1(E_{\si_0};\Z)$. We choose a basis such that $m=0$ and simply denote the matrices $M_{i,0}$ by $M_i$. The above matrices generate the entire modular group, because  it is well known
that the matrices in
${\rm SL}_2(\Z)$
\ben
S:=\begin{bmatrix} 1 & 1 \\ 0 & 1\end{bmatrix}=M_1,\quad
T:=\begin{bmatrix} 0 & 1 \\ -1 & 0\end{bmatrix}= M_3\,M_1\,
M_3\, =M_1\, M_3\, M_1
\een
are generators of ${\rm SL}_2(\Z)$. It is known that  ${\rm
  SL}_2(\Z)$ has a presentation in terms of the free group on two
generators $a$ and $b$ satisfying the relations 
\ben
a b a=b a b,\quad (aba)^4=1.
\een
It follows that ${\rm Ker}\, (\rho_0)$ is the normal subgroup of the
free group on 3 generators $M_1,M_2$, and $M_3$ generated by the relations
\beq\label{ker-rho0}
M_1M_3M_1=M_3M_1M_3,\quad (M_1M_3M_1)^4=1,\quad M_1M_3=M_2M_1.
\eeq

\subsubsection{Monodromy of the twisted sector}
We fix a basis of monomials
$\phi_\r(\x)=x_1^{r_1}x_2^{r_2}x_3^{r_3}$ in the twisted sector of
$H$, where $\r=(r_1,r_2,r_3)$ is given by
\ben
\r=\{ (1,0,0),(2,0,0),(0,1,0),(1,1,0),(3,0,0),(2,1,0),(4,0,0),(5,0,0)\}.
\een
The Picard-Fuchs equations for $\Phi_\r(\si),$  $\r=(1,0,0), (5,0,0)$
have order 1
\ben
(27+4\si^3)\d_\si \, \Phi_\r = 12\si^2\, {\deg}(\phi_\r)\,
\Phi_\r.
\een
The monodromy of these equations is straightforward to compute:
\ben
\Phi_\r\mapsto e^{2\pi\sqrt{-1}\,{\deg}(\phi_\r)}\,  \Phi_\r
\een
The remaining periods satisfy three systems of differential equations.  Let
us describe them and their monodromies. In all three cases we use the
substitution $x=-4\si^3/27$ to reduce the system to a hypergeometric
equation and then the monodromy is computed in the same way as before. We recall from \cite{MS} that the first system is
\ben
\left\{
\begin{aligned}
(27+4\si^3)\d_\si \, \Phi_{k+2,0,0} & = -(k+3)\si^2\, \Phi_{k+2,0,0} -\frac{9(k+1)}{2}\Phi_{k,1,0} \\
(27+4\si^3)\d_\si \, \Phi_{k,1,0} & =  \frac{3(k+3)\si}{2}\Phi_{k+2,0,0} -  (k+1)\si^2 \Phi_{k,1,0} 
\end{aligned}
\right.
\een  
The solutions to the system have the following form:
\beq\label{fund-solution}
\begin{bmatrix} \Phi_{k+2,0,0} \\ \Phi_{k,1,0}\end{bmatrix} = 
\begin{bmatrix} 
   F_1^{(1)}  &    F_2^{(1)}\\
L_{k} F_1^{(1)}  &  L_{k} F_2^{(1)}
\end{bmatrix}
\,  \begin{bmatrix} A_{k+2,0,0} \\ A_{k,1,0}\end{bmatrix} 
\eeq
where $A_{k+2,0,0}$ and $A_{k,1,0}$ are flat sections of the vanishing cohomology
bundle and $L_k$ is a first order differential operator, 
\ben
L_{k}:=\frac{2}{3(k+3)\si}\left( (27+4\si^4)\d_{\si}\Phi_{k,1,0}+(k+1)\si^2\Phi_{k,1,0}\right),\quad k=0,1,2
\een 
The period $\Phi_{k+2,0,0}$ satisfies the hypergeometric equation with
weights 
$$(\alpha_k,\beta_k,\gamma_k)=(1/4,3/4,2/3),\quad (1/3,5/6,2/3),\quad (5/12,11/12,2/3), \quad k=0,1,2.$$
It is convenient to decompose $\lieh_{\neq 0}$ into eigenspaces $\lieh_d$ of the classical monodromy 
$ J$, where the eigenvalue corresponding to $\lieh_d$ is $e^{2\pi\sqrt{-1}d}$.
The monodromy representation $\rho_{\neq 0}$ splits accordingly into a direct sum of representations $\rho_d$ 
and we denote by $M_{i,d}=[\rho_d(\gamma_i)]$ ($i=1,2,3$) the matrices of the monodromy representation in an
appropriately chosen basis of $\lieh_d$.  
The same argument as in the $E_7^{(1,1)}$-case gives that 
\ben
M_{1,1/3} = 
\begin{bmatrix} 1 & 0\\ 0 & e^{-2\pi \sqrt{-1}/3}\end{bmatrix},\quad 
M_{2,1/3} = 
\begin{bmatrix} -\frac{i}{\sqrt{3}} &-\frac{i}{2\sqrt{3}}  \\ 
-2+\frac{2i}{\sqrt{3}}
& \frac{1}{2}-\frac{i}{2\sqrt{3}}\end{bmatrix},
\quad
M_{3,1/3} = 
\begin{bmatrix} -\frac{i}{\sqrt{3}} & -\frac{1}{4}+\frac{i}{4\sqrt{3}} \\ 
-\frac{4i}{\sqrt{3}} &
\frac{1}{2}-\frac{i}{2\sqrt{3}}\end{bmatrix}.
\een
Using computer software one can check that the group generated by
these matrices is finite of order 24. Alternatively, since the exponents of the
hypergeometric equation are
\ben
\la_0=1-\gamma=\frac{1}{3},\quad \la_1=\gamma-\alpha-\beta =
\frac{1}{3},\quad
\la_\infty = \beta-\alpha=\frac{1}{2},
\een
we find that the hypergeometric equation is in the Schwarz list and
that its monodromy group is known to be isomorphic to $A_4\times
\Z/2$. It has order 24. 

From the system for $(\Phi_{300},\Phi_{110})$ we get 
\ben
M_{1,1/2} = 
\begin{bmatrix} 1 & 0\\ 0 & e^{-2\pi \sqrt{-1}/2}\end{bmatrix},\quad 
M_{2,1/2} = 
\begin{bmatrix} -\frac{1}{2} & -\frac{i}{2\sqrt{3}} \\ 
\frac{3\sqrt{3}i}{2} & \frac{1}{2}
\end{bmatrix},
\quad
M_{3,1/2} = 
\begin{bmatrix} -\frac{1}{2} & \frac{i}{2\sqrt{3}} \\ 
 -\frac{3\sqrt{3}i}{2}&
\frac{1}{2}\end{bmatrix}.
\een
It is easy to check that these matrices generate a group with $6$
elements. Again, we can obtain this from the general theory, since the exponents  are
\ben
\la_0=1-\gamma=\frac{1}{3},\quad \la_1=\gamma-\alpha-\beta =
-\frac{1}{2},\quad
\la_\infty = \beta-\alpha=\frac{1}{2}
\een
and the hypergeometric equation is again in the Schwarz list. The
monodromy group is known to be the dihedral group $D_6$. 

Finally, from the system for $(\Phi_{400},\Phi_{210})$ we get
 \ben
M_{1,2/3} = 
\begin{bmatrix} 1 & 0\\ 0 & e^{-4\pi \sqrt{-1}/3}\end{bmatrix},\quad 
M_{2,2/3} = 
\begin{bmatrix} 
\frac{i}{\sqrt{3}} & -\frac{\sqrt{3}i}{8}  \\ 
\frac{8}{3}+\frac{8\sqrt{3}i}{9}& \frac{1}{2}+\frac{i}{2\sqrt{3}}\end{bmatrix}, \quad
M_{3,2/3} = 
\begin{bmatrix} 
\frac{i}{\sqrt{3}} & \frac{3}{16}+\frac{\sqrt{3}i}{16} \\ 
-\frac{16i}{3\sqrt{3}} &\frac{1}{2}+\frac{\sqrt{3}i}{6}  
\end{bmatrix}.
\een
Again the matrices generate a group of order 24, which can also be proved by the general theory, since the exponents are
\ben
\la_0=1-\gamma=\frac{1}{3},\quad \la_1=\gamma-\alpha-\beta =
-\frac{2}{3},\quad
\la_\infty = \beta-\alpha=\frac{1}{2}.
\een
Again, the monodromy group is isomorphic to $A_4\times \Z/2$. One can check that
\beq\label{intertwiner}
M^{-1}_{1,2/3} = S\, M_{1,1/3}\, S^{-1},\quad 
M^{-1}_{2,2/3} = S\, M_{3,1/3}\, S^{-1},\quad 
M^{-1}_{3,2/3} = S\, M_{2,1/3}\, S^{-1},
\eeq
where $S$ is the diagonal matrix with diagonal entries $-4/3$ and $1$.
\begin{lemma}\label{gamma6}
The condition \eqref{kernels} is satisfied and $\Gamma(W)=\Gamma(6)$. 
\end{lemma}
\proof
For the first part of the Lemma, it is enough to check that the
matrices $M_{i,d}$, $1\leq i\leq 3$, $d=1/3,1/2$ satisfy the relations
\eqref{ker-rho0}. Note that due to \eqref{intertwiner} $M_{i,2/3}$
also satisfy the relations. This is a straightforward computation,
which we omit. 

The homomorphism $\rho_W$ decomposes naturally into
$(\rho_{W,1/3},\rho_{W,1/2},\rho_{W,2/3})$, where 
\ben
\rho_{W,d}: \widetilde{\Gamma}(W)\to {\rm GL}(\, \lieh_d \, )/\mu_6,\quad d=1/3,1/2,2/3.
\een
We claim that ${\rm Ker}(\rho_{W,1/3})={\rm Ker}(\rho_{W,2/3})=\Gamma(3).$ 
It is easy to find that $\Gamma(3)$ is generated by
the following matrices:
\ben
\begin{bmatrix} 1 & 3 \\ 0 & 1\end{bmatrix},\quad 
\begin{bmatrix}  1 & 0 \\ -3 & 1\end{bmatrix},\quad
\begin{bmatrix} -2 & 3 \\ -3 & 4\end{bmatrix}\, .
\een 
We leave it to the reader to check that they belong to the kernel of
$\rho_{W,d}$ for $d=1/3$ and $d=2/3$. Our claim follows, because the index of $\Gamma(3)$ in
${\rm SL}_2(\Z)$ is 24 (see \cite{Mi}). 
Similarly one can verify that ${\rm Ker}\, (\rho_{W,1/2})=\Gamma(2)$, so we get
\ben
\Gamma(W) ={\rm Ker}\ (\rho_W)={\rm Ker}\, (\rho_{W,1/3})\cap  {\rm Ker}\, (\rho_{W,1/2})\cap {\rm
  Ker}\, (\rho_{W,2/3})= \Gamma(3)\cap \Gamma(2)= \Gamma(6).
\qed
\een

\end{document}